\documentclass[11pt]{amsart}
\usepackage{amsmath,amssymb,amsthm,latexsym,cite,cancel}
\usepackage[small]{caption}
\usepackage{graphicx,wasysym,overpic,tikz,color}
\usepackage{mathrsfs,subfigure,color}
\usepackage{cite}
\usepackage[colorlinks=true,urlcolor=blue,
citecolor=red,linkcolor=blue,linktocpage,pdfpagelabels,
bookmarksnumbered,bookmarksopen]{hyperref}
\usepackage[italian,english]{babel}
\usepackage{units}
\usepackage{enumitem,color}
\usepackage[left=2.1cm,right=2.1cm,top=2.71cm,bottom=2.71cm]{geometry}
\usepackage[hyperpageref]{backref}
\usepackage[colorinlistoftodos]{todonotes}
\newtheorem{theorem}{Theorem}[section]
\newtheorem{definition}[theorem]{Definition}
\newtheorem{proposition}[theorem]{Proposition}
\newtheorem{Lemma}[theorem]{Lemma}
\newtheorem{remark}[theorem]{Remark}

\newtheorem{corollary}[theorem]{Corollary}

\numberwithin{equation}{section}

\title[$(2, q)$-Laplacian equations]{Normalized solutions to a class of $(2, q)$-Laplacian equations\\in the strongly sublinear regime}

\author[R. Ding]{Rui Ding}

\address[R. Ding]{\newline\indent
	School of Mathematics
	\newline\indent
	East China University of Science and Technology
	\newline\indent
	Shanghai 200237, PR China }
\email{\href{mailto:dingrui18363922468@outlook.com }{dingrui18363922468@outlook.com }}

\author[C. Ji]{Chao Ji}

\address[C. Ji]{\newline\indent
	School of Mathematics
	\newline\indent
	East China University of Science and Technology
	\newline\indent
	Shanghai 200237, PR China }
\email{\href{mailto:jichao@ecust.edu.cn}{jichao@ecust.edu.cn}}

\author[P. Pucci]{Patrizia Pucci}

\address[P. Pucci]{\newline\indent
	Dipartimento di Matematica e Informatica
	\newline\indent
	 Universit\`{a} degli Studi di Perugia
	\newline\indent
	06123, Italy}
\email{\href{mailto: patrizia.pucci@unipg.it}{patrizia.pucci@unipg.it}}
\date{\today}

\begin{document}
	\maketitle
	\begin{abstract}
		In this paper, we consider the existence and multiplicity of normalized solutions for the following $(2, q)$-Laplacian equation
		\begin{equation}\label{Equation1}
			\left\{\begin{aligned}
				&-\Delta u-\Delta_q u+\lambda u=g(u),\quad x \in \mathbb{R}^N, \\
				&\int_{\mathbb{R}^N}u^2 d x=c^2, \\
			\end{aligned}\right.
\tag{$\mathscr E_\lambda$}
\end{equation}
where $1<q<N$,  $\Delta_q=\operatorname{div}\left(|\nabla u|^{q-2} \nabla u\right)$ is the $q$-Laplacian operator, $\lambda$ is a Lagrange multiplier and $c>0$ is a constant. The nonlinearity $g:\mathbb{R}\rightarrow \mathbb{R}$ is continuous and the behaviour of $g$ at the origin is allowed to be strongly sublinear, i.e., $\lim \limits _{s \rightarrow 0} g(s) / s=-\infty$, which includes the logarithmic nonlinearity
		$$
		g(s)= s \log s^2.
		$$
		We consider a family of approximating problems that can be set in $H^1\left(\mathbb{R}^N\right)\cap D^{1, q}\left(\mathbb{R}^N\right)$ and the corresponding least-energy solutions.  Then, we prove that such a family of solutions converges to a least-energy solution to the original problem. Additionally, under certain assumptions about $g$ that allow us to work in a suitable subspace of $H^1\left(\mathbb{R}^N\right)\cap D^{1, q}\left(\mathbb{R}^N\right)$, we prove the existence of infinitely many solutions of the above $(2, q)$-Laplacian equation.
	\end{abstract}

	\section{Introduction}
	In this paper, we are concerned with the existence and multiplicity of normalized solutions to the following $(2, q)$-Laplacian equation
\begin{equation}\label{Equation}
	\left\{\begin{array}{l}
		-\Delta u-\Delta_q u+\lambda u=g(u),{\quad x \in \mathbb{R}^N, } \\
		\int_{\mathbb{R}^N}u^2 d x=c^2, \\
	\end{array}\right.\tag{$\mathscr E_\lambda$}
\end{equation}
where $\Delta_q u=\operatorname{div}\left(|\nabla u|^{q-2} \nabla u\right)$ is the $q$-Laplacian of $u$, $u \in X$, with $X:=H^1\left(\mathbb{R}^N\right) \cap D^{1, q}\left(\mathbb{R}^N\right)$, $c>0$, $\lambda \in \mathbb{R}$ is an unknown parameter that appears as a Lagrange multiplier, $ \frac{2 N}{N+2}<q<2$,
$N \geq 2 $ or $ 2<q<N$, $N \geq 3$.

In recent years, the  $(p, q)$-Laplacian equation has received considerable
attention. The $(p, q)$-Laplacian equation comes from
the general reaction-diffusion equation
	\begin{equation}\label{Eq-Reaction}
		u_t=\operatorname{div}(D(u) \nabla u)+f(x, u) \text { where } D(u):=|\nabla u|^{p-2}+|\nabla u|^{q-2} \text {, }
	\end{equation}
this equation has a wide range of applications in physics and related sciences,
such as plasma physics, biophysics, and chemical reaction design. In such applications,
the function $u$ describes a concentration; $\operatorname{div}(D(u) \nabla u)$ corresponds to the diffusion and $f(x, u)$ is the reaction related to source and loss processes, for more details, please refer to \cite{Cherfils}.

Taking the stationary version of \eqref{Eq-Reaction}, with $p=2$, we obtain the $(2, q)$-Laplacian equation
	\begin{equation}\label{Eq-Equationnomass}
		-\Delta u-\Delta_q u=f(x, u), \quad x\in  \mathbb{R}^N.
\end{equation}
In this paper, inspired by the fact that physicists are often {interested in} normalized solutions, we look for solutions of \eqref{Eq-Equationnomass} in $X$ having a prescribed $L^2$-norm. This approach seems to be particularly meaningful from the physical point of view, because in nonlinear optics and {in} the theory of Bose-Einstein condensates, there is a conservation of mass, see \cite{Frantzeskakis2010BEC,Malomed2008BEC}. For the results involving  $(p, q)$-Laplacian equations, we can refer to \cite{LP, MV}.
	
Due to our scope, we would like to mention \cite{Baldelli20222q}, where
Baldelli and Yang studied the existence of normalized solutions of the following $(2, q)$-Laplacian equation for all possible cases
depending on the value of $ p $,
\begin{equation}\label{eqBaldelli}
	\left\{\begin{array}{l}
		-\Delta u-\Delta_q u=\lambda u+|u|^{p-2} u, \quad x \in \mathbb{R}^N, \\
		\int_{\mathbb{R}^N}|u|^2 d x=c^2 .
	\end{array}\right.
\end{equation}
In the $L^2$-subcritical case, they studied a global minimization problem and obtained a ground state solution for problem \eqref{eqBaldelli}. In the $L^2$-critical case, they proved several nonexistence results, which were also extended to the $L^q$-critical case. Finally, for the $L^2$-supercritical case, they proved the existence of a ground state and infinitely many
radial solutions.

In a recent work \cite{Lcai2024Norm}, Cai and R{\u{a}}dulescu studied the following $(p, q)$-Laplacian equation with $L^p$-constraint
\begin{equation}\label{eqLcai}
	\left\{\begin{array}{l}
		-\Delta_p u-\Delta_q u+\lambda|u|^{p-2} u=f(u), \quad x \in \mathbb{R}^N, \\
		\int_{\mathbb{R}^N}|u|^p d x=c^p, \\
		u \in W^{1, p}\left(\mathbb{R}^N\right) \cap W^{1, q}\left(\mathbb{R}^N\right).
	\end{array}\right.
\end{equation}
For problem \eqref{eqLcai},
they established the existence of ground states, and revealed some basic behaviors of the ground state energy $E_c$ as $c>0$ varies. The analysis developed in~\cite{Lcai2024Norm} allows {them} to provide the general growth assumptions imposed to the reaction $f$.

On the other hand,
{in the past decades}, the following logarithmic Schr\"odinger equation has received considerable attention
\begin{equation}\label{E_var}
	-{\varepsilon}^2\Delta u+ V(x)u=u \log u^2, \quad x \in \mathbb{R}^N,
\end{equation}
where $\varepsilon >0$  and $V:\mathbb{R}^{N}\rightarrow \mathbb{R}$ is a potential function. {Equation} \eqref{E_var} has some physical applications, such as quantum mechanics, quantum optics, nuclear physics, transport, diffusion phenomena, and others. For further details, we refer to~\cite{Zlo}. {Equation} \eqref{E_var} also raises many difficult mathematical {questions}. For example, there exists $u \in H^{1}(\mathbb{R}^N)$ such that $\int_{\mathbb{R}^N}u^{2}\log u^2 \, dx=-\infty$. Thus, the {underlying} energy functional associated
to  \eqref{E_var} is no longer  of class $C^{1}$. In order to overcome this technical difficulty some authors have used different techniques to establish  the existence, multiplicity and concentration of the solutions under some assumptions on the potential {$V(x)$}, which can be seen in  \cite{AlvesChaobefore, AJ0, AJ1, cs, sz, sz2, WZ} and the references therein. {See} also \cite{AJ2} for the normalized solutions of logarithmic Schr\"{o}dinger equations.

In particular, {in \cite{Mederski2023sublinear}},
Mederski and Schino investigated the least-energy normalized solutions
{of}  the following Schr\"odinger equation
$$
-\Delta u+\lambda u=g(u), \quad x \in \mathbb{R}^N,
$$
coupled with the mass constraint $\int_{\mathbb{R}^N}|u|^2 d x=\rho^2$, {when} $N \geq 2$.
{{In~\cite{Mederski2023sublinear}},} the behaviour of $g$ at the origin is allowed to be strongly sublinear, i.e., $\lim \limits _{s \rightarrow 0} g(s) / s=-\infty$, which includes the logarithmic nonlinearity represented as the special case
\begin{equation}\label{Log}
	g(s)=s \log s^2.
\end{equation}
Under some assumptions about $g$, {Mederski and Schino
in~\cite{Mederski2023sublinear}}  proved the existence of infinitely many normalized solutions {of} the above problem.
In addition, when
\begin{equation}\label{Logp}
	g(s)=\alpha s \ln s^2+\mu|u|^{p-2} u,
\end{equation}
under certain conditions, they provided the non-existence of solutions.

Motivated by the {results} mentioned above,  this paper aims to investigate the existence and multiplicity of normalized solutions to the $(2, q)$-Laplacian equation \eqref{Equation},
where $g$ is allowed to be strongly sublinear at the origin. Before stating the main results of this paper, we present the assumptions required on~$g$. 	
Define $G(s)=\int_0^s g(t) d t$ and $\bar{q}:=\left(1+\frac{2}{N}\right) \max \{2, q\}$, ${q^{\prime}}:=\max \left\{2^*, q^*\right\}$
and $s^*:=\frac{s N}{N-s}$. Let
\begin{equation}\label{G+}
	G_{+}(s)= \begin{cases}\int_0^s \max \{g(t), 0\} d t, \quad &\text { if } s \geq 0, \\
\int_s^0 \max \{-g(t), 0\} d t, \quad &\text { if } s<0 .\end{cases}
\end{equation} \\
$(g_0)$ $g: \mathbb{R} \rightarrow \mathbb{R}$ is continuous and $g(0)=0$.\\
$(g_1)$ $\lim \limits _{s \rightarrow 0} G_{+}(s) /|s|^2=0$.\\
$(g_2)$ If $N \geq 3$, then $\limsup\limits _{|s| \rightarrow +\infty}|g(s)| /|s|^{q^{\prime}-1}<+\infty$.\\
If $N=2$, then $\lim \limits _{|s| \rightarrow +\infty} g(s) / e^{\alpha s^2}=0$ for all $ \alpha>0$.\\
$(g_3)$ $\lim \limits _{|s| \rightarrow +\infty} G_{+}(s) /|s|^{\bar{q}}=0$.\\
$(g_4)$ There exists $\xi_0 \neq 0$ such that $G\left(\xi_0\right)>0$.\\
Let us define
$$
\mathcal{D}(c):=\left\{u \in X: \int_{\mathbb{R}^N}|u|^2 d x \leq c^2\right\} \quad \text { and } \quad \mathcal{S}(c):=\left\{u \in X: \int_{\mathbb{R}^N}|u|^2 d x=c^2\right\} .
$$
{Let
$J: X\rightarrow \mathbb{R} \cup\{\infty\}$ be the energy
functional associated with \eqref{Equation}
\begin{equation}\label{functional}
	J(u)=\frac{1}{2} \int_{\mathbb{R}^N}|\nabla u|^2 d x+\frac{1}{q} \int_{\mathbb{R}^N}|\nabla u|^q d x-\int_{\mathbb{R}^N} G(u) d x ,
\end{equation}
on the constraint $\mathcal{S}(c)$.}
The  energy functional $J$ is not well defined in $X$. Indeed, when  $g(u)=u \log u^2$,   there exists $u \in X$ such that $\int_{\mathbb{R}^N} u^2 \log u^2 d x=-\infty$, as shown in  details in the Appendix. Furthermore, the reason of  the restriction $ \frac{2 N}{N+2}<q$ when $N\ge2$
appears in Remark~\ref{Re-2N/N+2} of the Appendix.
	
Our first result is the existence of the normalized solutions to \eqref{Equation}, with additional properties. Specifically, we define  the normalized solution to \eqref{Equation} as a pair $(u, \lambda) \in X \times \mathbb{R}$ such that $J^{\prime}(u) v=-\lambda u v$ for every $v \in \mathcal{C}_0^{\infty}\left(\mathbb{R}^N\right)$.
\begin{theorem}\label{theorem1}
	If $g$ satisfies $(g_0)-(g_4)$, then there exists $\bar{c}\geq0$ such that for every $c>\bar{c}$, there exist $\lambda>0$ and $u \in \mathcal{S}(c)$ such that $J(u)=\min _{\mathcal{D}} J<0$ and $(u, \lambda)$ is a solution to \eqref{Equation}. Moreover, $u$ has constant sign and is, up to a translation, radial and radially monotonic.
\end{theorem}
{\begin{remark}
 The assumptions $ (g_0)-(g_4) $ include the classical case $\lim _{s \rightarrow 0} G(s) /|s|^2=0$, which implies that $J$ is of class $\mathcal{C}^1$ in $X$. In this case, If $u$ is a critical point of $\left.J\right|_{\mathcal{S}(c)}$, then there exist $\lambda\in\mathbb{R}$ such that $(u, \lambda)$ is a solution to \eqref{Equation}. An important example is $G(s)=|s|^p$, where $p\in(2,\tilde{q})$. For more details, we refer to \cite{Baldelli20222q}.
\end{remark}
It is a natural question to ask ``When $\bar{c}=0$ holds". To answer this, the behavior of $g$ near $ 0 $ is critical. We can establish the following results:
\begin{proposition}\label{prop1}
Suppose $g$ satisfies assumptions $\left(g_0\right)-\left(g_4\right)$, and define $\tilde{q}=\left(1+\frac{2}{N}\right) \min \{2, q\}$, then\\
(i) If $\liminf_{s \rightarrow 0}G(s)/s^{\tilde{q}}=+\infty$, then $\bar{c}=0$ in Theorem \ref{theorem1}.\\
(ii) If $\limsup_{s \rightarrow 0} G(s) / s^{\tilde{q}}<+\infty$, then $\bar{c}>0$ in Theorem \ref{theorem1}.	
\end{proposition}
Our next result concerns the limit of the ground energy map $c \mapsto \inf _{\mathcal{D}(c)} J$ as $c \rightarrow +\infty$.
\begin{proposition}\label{prop2}
If $g$ satisfies  assumptions $ (g_0)-(g_4) $, then
$$
\lim _{c \rightarrow +\infty} \inf _{\mathcal{D}(c)} J=-\infty.
$$
\end{proposition}
When $g$ is as in \eqref{Logp}, we also have the following non-existence result.
\begin{theorem}\label{theorem2}
Let $g$ be given by \eqref{Logp} with $\alpha>0, \mu<0$, and $2<p \leq q^{\prime}$ if $N\geq3$, $p>2$ if $N=2$. A solution $(u, \lambda) \in X \times(0, \infty)$ to \eqref{Equation} such that $G(u) \in L^1\left(\mathbb{R}^N\right)$ exists if and only if $-\frac{\alpha p}{p-2} e^{-p / 2}<\mu$.
\end{theorem}}
{Equation \eqref{Equation}} presents several challenges.  On one hand, for the $(2,q)$-Laplacian operator, due to the appearance of $q$-Laplacian term, the working space $X$ is not an {Hilbertian} space. {This} makes some estimates more difficult and complicated. For example, it is challenging to apply the {Br\'ezis}-Lieb lemma to bounded Palais-Smale sequences.  This issue is addressed in the proof of Lemma~\ref{LemmaStrongconver} in Section~\ref{Sec3}. On the other hand, {the nonlinearity $g$} is strongly sublinear near zero, this in particular implies that the corresponding functional  $J$ is not well-defined {in} the {working} space $X$. To handle this, we consider a series of perturbed functionals as {in \eqref{Functionalperturbed} below}. Then, in
Lemma~\ref{LemmaBoundedbelow}, we are able to prove that the solutions to these perturbed problems are bounded in $X$ with respect to $\varepsilon$. Besides, their weak limits are minimizers of the original functional in the set $\mathcal{D}(c)$. Importantly, the solutions to the perturbed problems belong to the set $\mathcal{S}(c)$, which is crucial when proving that the solutions {of} the original problem also belong to $\mathcal{S}(c)$.

	Given functions $f_1,\, f_2: \mathbb{R} \rightarrow \mathbb{R}$, we introduce the following notation:
$f_1\lesssim f_2$ provided that $f_1(s) \leq Cf_2(s)$
for  all $s\in\mathbb R$, where $C$ is a positive constant independent of $s$.
	
	Let $g_{+}(s)=G_{+}^{\prime}(s)$, $g_{-}(s):=g_{+}(s)-g(s)$, and $G_{-}(s):=G_{+}(s)-G(s) \geq 0$ for $s \in \mathbb{R}$.  Thus, $G_{+}(u) \in L^1\left(\mathbb{R}^N\right)$ for all $u \in X$
in view of $(g_1)$ and $(g_3)$. However, $G_{-}(u)$ may not be integrable unless $G_{-}(s) \lesssim|s|^2$ for small $|s|$. In order to overcome this difficulty, for every $\varepsilon \in(0,1)$, let us take an even
	function $\varphi_{\varepsilon}: \mathbb{R} \rightarrow[0,1]$ such that $\varphi_{\varepsilon}(s)=|s| / \varepsilon$ for $|s| \leq \varepsilon, \varphi_{\varepsilon}(s)=1$ for $|s| \geq \varepsilon$. Now we introduce a new perturbed functional
	\begin{equation}\label{Functionalperturbed}
		J_{\varepsilon}(u)=\frac{1}{2} \int_{\mathbb{R}^N}|\nabla u|^2 d x+\frac{1}{q} \int_{\mathbb{R}^N}|\nabla u|^q d x+\int_{\mathbb{R}^N} G_{-}^{\varepsilon}(u) d x-\int_{\mathbb{R}^N} G_{+}(u) d x,
	\end{equation}
where $G_{-}^{\varepsilon}(s)=\int_0^s \varphi_{\varepsilon}(t) g_{-}(t) d t, s \in \mathbb{R}$.
Since $G_{-}^{\varepsilon}(s) \leq c_{\varepsilon}|s|^2$ for every $|s| \leq 1$ and some constant $c_{\varepsilon}>0$ depending only on $\varepsilon>0$, clearly
$J_{\varepsilon}$  is of class $\mathcal{C}^1(X)$
for any $\varepsilon \in(0,1)$.
{
\begin{remark}
	When $N=2$, it is natural to extend  assumption $(g_2)$ to the nonlinearity with exponential critical growth at infinity
	\begin{equation*}
		 \lim_{|s| \to \infty} \frac{g(s)}{e^{\alpha s^2}} = 0, \quad \text{for all} \,\, \alpha > 4\pi.
	\end{equation*}
	Under the above assumption, we can indeed prove
	the existence of a Palais-Smale sequence for $\left.J_{\varepsilon}\right|_{\mathcal{D}(c)}$, but the convergence of such sequence in $X$ is a very delicate problem, which at the moment we could not solve.
\end{remark}
}
	Now, we turn  {back to the main problem} of finding multiple normalized solutions (in fact, infinitely many)
{for}~\eqref{Equation}. {To this aim}, we need to modify both the assumptions {on} $g$ and our approach.

We assume that the right-hand side in \eqref{Equation} is given {in the form}
\begin{equation}\label{Eqg2}	
	g(s)=f(s)-a(s),
\end{equation}
together with the following assumptions {on $a$ and $f$}, where $F(s)=\int_0^s f(t) d t$, $A(s)=\int_0^s a(t) d t$, $ F_{+}$ is defined via \eqref{G+} replacing $g$ with $f$, and $f_{+}=F_{+}^{\prime}$.\\
(A) $A \in \mathcal{C}^1(\mathbb{R})$ is an $N$-function that satisfies the $\Delta_2$ and $\nabla_2$ conditions globally and such that $\lim \limits _{s \rightarrow 0} A(s) / s^2=\infty$ and $s \mapsto a(s) s$ is convex {in $\mathbb R$}.\\
$(f_0)$ $f: \mathbb{R} \rightarrow \mathbb{R}$ is continuous and odd.\\
$(f_1)$ $\lim \limits _{s \rightarrow 0} f_{+}(s) / s=0$.\\
$(f_2)$ If $N \geq 3$, then $|f(s)| \lesssim a(s)+|s|+|s|^{q^{\prime}-1}$, {again with ${q^{\prime}}:=\max \left\{2^*, q^*\right\}$.}\\
If $N=2$, then for all $m \geq 2$ and $\alpha>4 \pi$ there holds $|f(u)| \lesssim a(s)+|s|+|s|^{m-1}\left(e^{\alpha s^2}-1\right)$.
\\
$(f_3)$ $\lim \sup\limits _{|s| \rightarrow \infty} f_{+}(s) /|s|^{\bar{q}-1}<\infty$ and
\begin{equation}\label{eqeta}
	\eta:=\limsup _{|s| \rightarrow \infty} \frac{F_{+}(s)}{|s|^{\bar{q}}}<\infty .
\end{equation}
Let us define
\begin{equation}\label{W}
W:=\left\{u \in X: A(u) \in L^1\left(\mathbb{R}^N\right)\right\}.
\end{equation}
To present our results, we introduce some notations. For a subgroup $\mathcal{O} \subset \mathcal{O}(N)$,
\begin{equation}\label{WO}
W_{\mathcal{O}}:=\{u \in W: u(g \cdot)=u \text { for all } g \in \mathcal{O}\} .
\end{equation}
If $N=4$ or $N \geq 6$, in order to find non-radial solutions
{of  \eqref{Equation}, following  \cite{Mederski2020Nonlinearity,Mederski2023sublinear}, we
fix  $M$, with} $2 \leq M \leq N / 2$ such that $N-2 M \neq 1$ and {we put} $\tau(x)=\left(x_2, x_1, x_3\right)$ for $x=\left(x_1, x_2, x_3\right) \in \mathbb{R}^M \times \mathbb{R}^M \times \mathbb{R}^{N-2 M}$. Then, let
$$
{X}_\tau:=\left\{u \in W: u \circ \tau=-u\right\},	
$$
$$
\mathcal{X}:=\left\{u \in W_{\mathcal{O}(M) \times \mathcal{O}(M) \times \mathcal{O}(N-2 M)}: u \circ \tau=-u\right\}={X}_\tau\cap W_{\mathcal{O}(M) \times \mathcal{O}(M) \times \mathcal{O}(N-2 M)},
$$
here we agree that the components corresponding to $N-2 M$ do not exist when $N=2 M$ and observe that $\mathcal{X} \cap W_{\mathcal{O}(N)}=\{0\}$. For $2<p<2^*$,
{the number $C_{N, p}$ denotes} the best constant in the Gagliardo-Nirenberg inequality of Lemma \ref{LemmaGN}.

Our multiplicity result reads as follows.
\begin{theorem}\label{thm2muty}
	Let $g$ be {as in} \eqref{Eqg2}.
{Assume} that (A), $(f_0)-(f_3)$, and
	\begin{equation}\label{etaC}
			2(\eta+\delta) C_{N, \bar{q}}^{\bar{q}} c^{\bar{q} \left(1-\delta_{\bar{q}}\right)} <1
	\end{equation}
	hold, where $\eta$ is defined in \eqref{eqeta}. Then there exist infinitely many solutions $(u, \lambda) \in W_{\mathcal{O}(N)} \times \mathbb{R}$ to \eqref{Equation}, one of which say, $(\bar{u}, \bar{\lambda})$ having the property that $J(\bar{u})=\min _{\mathcal{S}(c) \cap W} {J(u)}$. If, in addition, $N=4$ or $N \geq 6$, then there exist infinitely many solutions $(v, \mu) \in \mathcal{X} \times \mathbb{R}$, one of which say, $(\bar{v}, \bar{\mu})$ having the property that $J(\bar{v})=\min _{\mathcal{S}(c) \cap \mathcal{X}} J$.
\end{theorem}

{Let us now present  the  main steps in
proving} Theorem \ref{thm2muty}. To prove the multiplicity results stated in Theorem \ref{thm2muty}, we work in the subspace $W_{\mathcal{O}}$ of $X$ where the functional $J$ is well-defined. From Corollary \ref{CorLoc}, we have the compact embedding $W_{\mathcal{O}} \hookrightarrow L^m\left(\mathbb{R}^N\right)$ for all $2\leq m<q^{\prime}$. This result, which is due to Lemma \ref{LemmaA(u)lions}, allows us to show that $\left.J\right|_{{W_{\mathcal{O}}} \cap \mathcal{S}(c)}$ satisfies the Palais-Smale condition, see Lemma \ref{LemA(u)PS}. Then, to obtain { Palais-Smale sequences of $J(u)$}, we make use of classical minimax arguments, see Theorem \ref{Thmmupt} proved in \cite{Jeanjean2019Nonradial}. Furthermore, {the proof of Theorem \ref{Thmmupt} is
based on the construction of suitable} mappings from $\mathbb{S}^{k-1}$ to  $W_{\mathcal{O}} \cap \mathcal{S}(c)$, {given} in Lemma \ref{LemGamma}, which relies on the existence of some special mappings {introduced} in \cite{Berestycki1983infinite}.

The paper is organized as follows. {In Section \ref{Sec2},} we introduce the functional setting and we give some preliminaries. {In Section~\ref{Sec3}, the perturbed problem {\eqref{Equationperturbed}}} was studied. Then, in Section \ref{Sec4}, we prove
Theorem~\ref{theorem1}.  In the last section, we focus on the multiplicity of normalized solutions  of~\eqref{Equation} and prove Theorem~\ref{thm2muty}.

\textbf{Notations:}
For $1 \leq p<\infty$ and $u \in L^p\left(\mathbb{R}^N\right)$, we denote $\|u\|_p:=\left(\int_{\mathbb{R}^N}|u|^p d x\right)^{\frac{1}{p}}$. The Hilbert space $H^1\left(\mathbb{R}^N\right)$ is defined as $H^1\left(\mathbb{R}^N\right):=\left\{u \in L^2\left(\mathbb{R}^N\right): \nabla u \in L^2\left(\mathbb{R}^N\right)\right\}$ with inner product $(u, v):=\int_{\mathbb{R}^N}(\nabla u \nabla v+u v) d x$ and norm $\|u\|:=\left(\|\nabla u\|_2^2+\|u\|_2^2\right)^{\frac{1}{2}}$. Similarly, $D^{1, q}\left(\mathbb{R}^N\right)$ is defined as $D^{1, q}\left(\mathbb{R}^N\right):=\left\{u \in L^{q^*}\left(\mathbb{R}^N\right): \nabla u \in L^q\left(\mathbb{R}^N\right)\right\}$ with the norm $\|u\|_{D^{1, q}\left(\mathbb{R}^N\right)}=\|\nabla u\|_q$.
Recalling $X=H^1\left(\mathbb{R}^N\right) \cap D^{1, q}\left(\mathbb{R}^N\right)$ endowed with the norm $\|u\|_X=\|u\|+\|\nabla u\|_q$. We use $" \rightarrow$ " and $" \rightharpoonup$ " to denote the strong and weak convergence in the related function spaces respectively. $C$ and $C_i$ will be positive constants which may depend on $ N $, $ p $ and $ q $ (but never on $ u $), {whose value is not relevant}. $\langle\cdot, \cdot\rangle$ denote the dual pair for any Banach space and its dual space. Finally, $o_n(1)$ and $O_n(1)$ mean that $\left|o_n(1)\right| \rightarrow 0$ and $\left|O_n(1)\right| \leq C$ as $n \rightarrow \infty$, respectively.

	\section{Preliminaries}\label{Sec2}
	In this section, we introduce some preliminary results that will be useful for proving our main results.

\begin{Lemma}[The Gagliardo-Nirenberg inequality, {\cite[Corollary 2.1]{Weinstein1983GN}}]\label{LemmaGN}
		When $N\geq3$,
		{let $ p \in$ $\left(2, q^*\right)$} and $\delta_p=\frac{N(p-2)}{2 p}$, then there exists a constant $C_{N, p}=\left(\frac{p}{2\left\|W_p\right\|_2^{p-2}}\right)^{\frac{1}{p}}>0$ such that
		\begin{equation}\label{Eq-GN inequality1}
			\|u\|_p \leq C_{N, p}\|\nabla u\|_2^{\delta_p}\|u\|_2^{\left(1-\delta_p\right)} \quad \forall u \in H^1\left(\mathbb{R}^N\right),
		\end{equation}
	where $W_p$ is the unique positive radial solution of $-\Delta W+\left(\frac{1}{\delta_p}-1\right) W=\frac{2}{p \delta_p}|W|^{p-2} W$ {in $\mathbb R^N$.}
	\end{Lemma}

		\begin{Lemma}[$L^q$-Gagliardo-Nirenberg inequality, {\cite[Theorem 2.1]{agueH3008sharp}}]\label{Lemma-GN inequality2}
		When $N\geq3$,
		let $q \in\left(\frac{2 N}{N+2}, N\right), p \in$ $\left(2, q^*\right)$ and $\nu_{p, q}=\frac{N q(p-2)}{p[N q-2(N-q)]}$. Then there exists a constant $K_{N, p}>0$ such that
		\begin{equation}\label{Eq-GN inequality2}
			\|u\|_p \leq K_{N, p}\|\nabla u\|_q^{\nu_{p, q}}\|u\|_2^{\left(1-\nu_{p, q}\right)}, \quad \forall u \in D^{1, q}\left(\mathbb{R}^N\right) \cap L^2\left(\mathbb{R}^N\right),
		\end{equation}	
		where
	{	$$
		\begin{gathered}
			K_{N, p}=\left(\frac{K}{\frac{1}{q}\left\|D W_{p, q}\right\|_q^q+\frac{1}{2}\left\|W_{p, q}\right\|_2^2}\right), \\
			K=(N q+p q-2 N) \left(\frac{[2(N q-p(N-q))]^{p(N-q)-N q}}{[q N(p-2)]^{N(p-2)}}\right)^{1 /[N q+p q-2 N]},
		\end{gathered}
		$$}
		and $W_{p, q}$ is the unique nonnegative radial solution of the following equation
		$$
		-\Delta_q W+W=\zeta|W|^{p-2} W,\,\,\, x\in \mathbb{R}^N
		$$
		where $\zeta=\|\nabla W\|_q^q+\|W\|_2^2$ is the Lagrangian multiplier.
	\end{Lemma}

	\begin{Lemma}[The Trudinger-Moser inequality, {\cite[Lemma 2.1]{Cao1992TM}}]\label{LemmaTrudinger-Moser}
		If $\alpha>0$ and $u \in H^1\left(\mathbb{R}^2\right)$, then
		$$
		\int_{\mathbb{R}^2}\left(e^{\alpha u^2}-1\right) \mathrm{d} x<\infty .
		$$
		Moreover, if $\|\nabla u\|_2 \leq 1,\|u\|_2 \leq M<\infty$ with $M>0$ and $0<\alpha<4 \pi$, then there exists a constant $C>0$, which depends only on $M$ and $\alpha$, such that
		$$
		\int_{\mathbb{R}^2}\left(e^{\alpha u^2}-1\right) \mathrm{d} x \leq C\left(M, \alpha\right) .
		$$
	\end{Lemma}

\begin{Lemma}[The Sobolev {Embedding} Theorem,{
\cite[Lemma 1.2]{Lcai2024Norm}} and \cite{Baldelli2023Born-Infeld}]\label{LemmaEmbed}
	The space $X$ is embedded continuously into $L^m\left(\mathbb{R}^N\right)$ for $m \in\left[{2}, q^{\prime}\right]$ and compactly into $L_{\text {loc }}^m\left(\mathbb{R}^N\right)$ for $m \in\left[1, q^{\prime}\right)$, $q^{\prime}:=\max\{2^*,q^*\}$.
 Denote $X_{\mathrm{rad}}:=\{u \in X : u \text{ is radially symmetric}\}$, then the space $X_{\mathrm{rad}}$ is embedded compactly into $L^m\left(\mathbb{R}^N\right)$ for $m \in \left(2, q^{\prime}\right)$.
\end{Lemma}
	
For the next lemma, we can take a similar argument as that of the classical {Concentration-Compactness principle}. See, for instance,\cite[Lemma 1.21]{Willem1996Minimax}.

 \begin{Lemma}\label{LemmaLions}
	Let $r>0$. {If $(u_n)_n$} is a bounded sequence in $X$ which satisfies
		$$
		\sup _{x \in \mathbb{R}^N} \int_{B_r(x)}\left|u_n\right|^2 d x \rightarrow 0, \quad \text { as } n \rightarrow \infty,
		$$
		then,
		$$
		\left\|u_n\right\|_m \rightarrow 0 \quad \text { as } n \rightarrow \infty
		$$
		holds for any $m \in\left(2,q^{\prime}\right)$, where ${q^{\prime}}=\max \left\{2^*, q^*\right\}$.
	\end{Lemma}

	\begin{Lemma}[Lemma 2.7, \cite{Li1994obstacle}]\label{LemmaConverae}
		Assume $s>1$. Let $\Omega$ be an open set in $\mathbb{R}^N, \alpha, \beta$ positive numbers and $a(x, \xi)$ in $C\left(\Omega \times \mathbb{R}^N, \mathbb{R}^N\right)$ such that\\
		(1) $\alpha|\xi|^s \leq a(x, \xi) \xi$ for all $(x, \xi) \in \Omega \times \mathbb{R}^N$,\\
		(2) $|a(x,\left.\xi)|\leq \beta| \xi|^{s-1}\right.$ for all $(x, \xi) \in \Omega \times \mathbb{R}^N$,\\
		(3) $(a(x, \xi)-a(x, \eta))(\xi- \eta)>0$ for all $(x, \xi) \in \Omega \times \mathbb{R}^N$ with $\xi \neq \eta$,\\
		(4) $a(x, \gamma \xi)=\gamma|\gamma|^{p-2} a(x, \xi)$ for all $(x, \xi) \in \Omega \times \mathbb{R}^N$ and $\gamma \in \mathbb{R} \backslash\{0\}$.\\
		Consider  {$(u_n)_n$},
$u \in W^{1, s}(\Omega)$, then $\nabla u_n \rightarrow \nabla u$ in $L^s(\Omega)$ if and only if
		$$
		\lim _{n \rightarrow \infty} \int_{\Omega}\Big(a\left(x, \nabla u_n(x)\right)-a(x, \nabla u(x))\Big)\left(\nabla u_n(x)-\nabla u(x)\right) d x=0 .
		$$
	\end{Lemma}
	
To conclude this section, we recall the following elementary inequality.
	This inequality will be used in Lemma \ref{LemmaStrongconver} to show that if $(u_n)_n$ is a minimizing sequence of $J$ and
	$u_n \rightharpoonup {u}$ in $X$, then
	$\nabla u_n \rightarrow \nabla {u}$ for a.e. $x\in \mathbb{R}^N$.
{	\begin{Lemma}[Remark 1, \cite{Baldelli20222qF}]\label{Remarkpost}
	There exists a constant $C(s)>0$ such that
	for all $x, y \in \mathbb{R}^N$ with $|x|+|y| \neq 0$,
	$$
		\left\langle|x|^{s-2} x-|y|^{s-2} y, x-y\right\rangle \geq C(s) \begin{cases}\dfrac{|x-y|^2}{(|x|+|y|)^{{(2-s)}}}, &1 \leq s<2, \\ |x-y|^s, & s \geq 2 .\end{cases}
		$$
\end{Lemma}}

	\section{The perturbed problem}\label{Sec3}

Define $m_{\varepsilon}(c)=\inf _{\mathcal{D}(c)} J_{\varepsilon}(u)$, where the functional $J_{\varepsilon}: X \rightarrow \mathbb{R}$ is given in equation  \eqref{Functionalperturbed}.  The main purpose of this section is to prove the following result.
	
	\begin{theorem}\label{theoremPerturbed}
		Assume that $g$ satisfies $(g_0)-\left(g_4\right)$. Then, for every $\beta>0$, there exists $\bar{c}\geq0$ such that for every $c>\bar{c}$ and $\varepsilon>0$, there exist $\lambda_{\varepsilon}>0$ and $u_{\varepsilon} \in X$ such that $J_{\varepsilon}\left(u_{\varepsilon}\right)=m_{\varepsilon}(c)<-\beta$ and
		\begin{equation}\label{Equationperturbed}
			\left\{\begin{array}{l}
				-\Delta u_{\varepsilon}-\Delta_q u_{\varepsilon}+\lambda_{\varepsilon}u_{\varepsilon} =g_{+}\left(u_{\varepsilon}\right)-\varphi_{\varepsilon}\left(u_{\varepsilon}\right) g_{-}\left(u_{\varepsilon}\right),\quad x \in \mathbb{R}^N, \\				\int_{\mathbb{R}^N}\left|u_{\varepsilon}\right|^2 d x=c^2 .
\end{array}\right.
\tag{$\mathscr P_\varepsilon$}
\end{equation}
		Moreover, every such $u_{\varepsilon}$ has constant sign and, up to a translation, is radial and radially monotonic.
	\end{theorem}
	We begin by proving a series of lemmas.
	
	\begin{Lemma}\label{LemmaBoundedbelow}
		Assume that $g$ satisfies $(g_0)-(g_3)$. Then, $\left.J_{\varepsilon}\right|_{\mathcal{D}(c)}$ is bounded from below.
		
	\end{Lemma}
	\begin{proof}
	Observing that
		$$
		\left\{\begin{array}{l}
			\bar{q} \delta_{\bar{q}}=2,\,\, \text { if } \frac{2 N}{N+2}<q<2, \\
			\bar{q} \delta_{\bar{q}}=q\left(1+\delta_q\right),\,\, \text { if } 2<q<N ,
		\end{array}\right.
		$$
		thus
		$$
		\bar{q} \delta_{\bar{q}}=\max \left\{2, q\left(1+\delta_q\right)\right\}, \,\,\text{for}\,\,\frac{2 N}{N+2}<q<2\,\,\text{and}\,\, 2<q<N.
		$$
		From $(g_1)$ and $(g_3)$, for any $\delta>0$, there exists $C_\delta>0$ such that for every $s \in \mathbb{R}$
		\begin{equation}\label{eq-G}
		G_{+}(s) \leq C_\delta|s|^2+\delta|s|^{\bar{q}} .
		\end{equation}
		Consequently, {if} $\frac{2 N}{N+2} < q < 2$, from Lemma \ref{LemmaGN}, for every $u \in \mathcal{D}(c)$, we have
	\begin{equation}\label{Eq-GN1}
		\begin{aligned}
			J_{\varepsilon}(u) & \geq \frac{1}{2}\|\nabla u\|_2^2-\int_{\mathbb{R}^N} G_{+}(u) d x \\
			& \geq \frac{1}{2}\|\nabla u\|_2^2-C_\delta\|u\|_2^2-\delta\|u\|_{\bar{q}}^{\bar{q}} \\
			& \geq \frac{1}{2}\|\nabla u\|_2^2-C_\delta\|u\|_2^2-\delta C_{N, \bar{q}}^{\bar{q}} c^{\bar{q}\left(1-\delta_{\bar{q}}\right)}\|\nabla u\|_2^{\bar{q} \delta_{\bar{q}}} \\
			& \geq\|\nabla u\|_2^2\left(\frac{1}{2}-\delta C_{N, \bar{q}}^{\bar{q}} c^{\bar{q}\left(1-\delta_{\bar{q}}\right)}\|\nabla u\|_2^{\bar{q} \delta_{\bar{q}}-2}\right)-C_\delta c^2 .
		\end{aligned}
	\end{equation}
{By choosing suitable $\delta$ with $\delta < \left(2 C_{N, \bar{q}}^{\bar{q}} c^{\bar{q} (1 - \delta_{\bar{q}})}\right)^{-1}$},  $\left.J_{\varepsilon}\right|_{\mathcal{D}(c)}$ is bounded from below.
Similarly,
if $2 < q <  N$, from Lemma \ref{Lemma-GN inequality2}, for every $u \in \mathcal{D}(c)$, we have
\begin{equation}\label{Eq-GN2}
	\begin{aligned}
		J_{\varepsilon}(u) & \geq \frac{1}{q}\|\nabla u\|_q^q-\int_{\mathbb{R}^N} G_{+}(u) d x \\
		& \geq \frac{1}{q}\|\nabla u\|_q^q-C_\delta\|u\|_2^2-\delta\|u\|_{\bar{q}}^{\bar{q}} \\
		& \geq \frac{1}{q}\|\nabla u\|_q^q-C_\delta\|u\|_2^2-\delta K_{N, \bar{q}}^{\bar{q}} \bar{c}^{\bar{q}\left(1-\nu_{\bar{q}, q}\right)}\|\nabla u\|_q^{\bar{q} \nu_{\bar{q}, q}} \\
		& \geq\|\nabla u\|_q^q\left(\frac{1}{q}-\delta K_{N, \bar{q}}^{\bar{q}} c^{\bar{q}\left(1-\nu_{\bar{q}, q}\right)}\|\nabla u\|_q^{\bar{q} \nu_{\bar{q}, q}-q}\right)-C_\delta c^2
	\end{aligned}
\end{equation}
 The same result can be easily obtained by {choosing suitable $\delta$ with} $\delta < \left(qK_{N, \bar{q}}^{\bar{q}} c^{\bar{q} (1 - \nu_{\bar{q},q})}\right)^{-1}$.
	\end{proof}
	
	\begin{Lemma}\label{Lemmam(c)<0}
		Assume that $g$ satisfies $(g_0)-(g_2)$ and $(g_4)$. Then for every $\beta>0$,  there exists $\bar{c}\geq0$ such that
		$$
		m_{\varepsilon}(c)<-\beta
		$$
		for every $c>\bar{c}$ and every $\varepsilon \in(0,1)$.
	\end{Lemma}
	\begin{proof}
{Take $R>0$ and} set
		$$
		u_R(x)= \begin{cases}\xi_0  & \text { if }|x| \leq R, \\ \xi_0 (R+1-|x|) & \text { if } R<|x| \leq R+1, \\ 0 & \text { if }|x|>R+1,\end{cases}
		$$
		where $\xi_0$ is {the} constant determined in $(g_4)$. It can be shown through direct calculation that $u_R \in X$.
			Arguing as in \cite[Lemma 2.3]{Shibata2013general}, for $R$ large enough,  $\int_{\mathbb{R}^N} G(u_R) d x \geq 1$.
{Putting  $$G^{\varepsilon}:=G_{+}-G_{-}^{\varepsilon} \geq G,$$ we see} that $\int_{\mathbb{R}^N} G^{\varepsilon}(u_R) d x \geq 1$ for any $\varepsilon \in(0,1)$. Now, define $u_c:=u_R\left(\|u_R\|_2^{2 / N} c^{-2 / N}x\right) \in \mathcal{S}(c)$,  then
		$$
		\begin{aligned}
			J_{\varepsilon}\left(u_c\right)
			& =\frac{1}{2} \int_{\mathbb{R}^N}|\nabla u_c|^2 d x+\frac{1}{q} \int_{\mathbb{R}^N}|\nabla u_c|^q d x-\int_{\mathbb{R}^N}G^{\varepsilon}\left(u_c\right) d x\\
			&={c^{\frac{2(N-2)}{N}}}\frac{\|\nabla u_R\|_2^2}{2\|u_R\|_2^{\frac{2(N-2)}{N}}}
			+{c^{\frac{2(N-q)}{N}}} \frac{\|\nabla u_R\|_q^q}{q\|u_R\|_2^{\frac{2(N-q)}{N}}}-
			c^2\frac{ \int_{\mathbb{R}^N} G^{\varepsilon}(u_R) d x }{\|u_R\|_2^2} \\
			& \leq {c^{\frac{2(N-2)}{N}}} \frac{\|\nabla u_R\|_2^2}{2\|u_R\|_2^{\frac{2(N-2)}{N}}}
			+{c^{\frac{2(N-q)}{N}}} \frac{\|\nabla u_R\|_q^q}{q\|u_R\|_2^{\frac{2(N-q)}{N}}}- c^2 \frac{1}{\|u_R\|_2^2}.
		\end{aligned}
$$
{Since $\frac{2(N-2)}{N}<2$ and $\frac{2(N-q)}{N}<2$,
$ J_{\varepsilon}\left(u_c\right) \rightarrow -\infty$
as $c \rightarrow \infty$}. {This concludes the proof.}
	\end{proof}
	
	\begin{Lemma}\label{LemmaSubadd}
		Assume that $g$ satisfies $(g_0)-(g_3)$. Then for any $c_1, c_2>0$, there holds $$m_{\varepsilon}\left(\sqrt{c_1^2+c_2^2}\right) \leq m_{\varepsilon}\left(c_1\right)+m_{\varepsilon}\left(c_2\right).$$
	\end{Lemma}
	\begin{proof}
		Since $m_{\varepsilon}(c)=\inf _{\mathcal{D}(c)} J_{\varepsilon}(u)$, from Lemma \ref{LemmaBoundedbelow}, it is easy to see that $m_{\varepsilon}(c)$ is finite for any $c>0$.
		 Fix
{$c_1,
\, c_2>0$ and $\delta>0$,} there exist $u_1, u_2 \in$ $\mathcal{C}_0^{\infty}\left(\mathbb{R}^N\right)$ such that $\Vert u_i\Vert_2 \leq c_i$ and $J_{\varepsilon}\left(u_i\right) \leq m_{\varepsilon}\left(c_i\right)+\delta$, where $c_{i}>0$ and $i=1,2$. In virtue of the translation invariance, we can assume that $u_1$ and $u_2$ have disjoint supports. Then $\left\|u_1+u_2\right\|_2^2=\left\|u_1\right\|_2^2+\left\|u_2\right\|_2^2 \leq c_1^2+c_2^2$ and so
		$$
		m_{\varepsilon}\left(\sqrt{c_1^2+c_2^2}\right) \leq J_{\varepsilon}\left(u_1+u_2\right)= J_{\varepsilon}\left(u_1\right)+J_{\varepsilon}\left(u_2\right)\leq m_{\varepsilon}\left(c_1\right)+m_{\varepsilon}\left(c_2\right)+2 \delta.
		$$
{Letting $\delta\rightarrow 0$, we get} $m_{\varepsilon}\left(\sqrt{c_1^2+c_2^2}\right) \leq m_{\varepsilon}\left(c_1\right)+m_{\varepsilon}\left(c_2\right)$.
	\end{proof}
	\begin{Lemma}\label{LemmaStrictSubadd}
	Assume that $g$ satisfies $(g_0)-(g_3)$. Then for any $c_1, c_2>0$, if there exist $u_i \in \mathcal{D}\left(c_i\right)$ such that $J_{\varepsilon}\left(u_i\right)=m_{\varepsilon}\left(c_i\right)$ for $ i=1,2$, and $\left(u_1, u_2\right) \neq(0,0)$, we have $$
	m_{\varepsilon}\left(\sqrt{c_1^2+c_2^2}\right)<m_{\varepsilon}\left(c_1\right)+m_{\varepsilon}\left(c_2\right).
	$$
	\end{Lemma}
	\begin{proof}
		First of all, {for any $s>0$}
$$\int_{\mathbb{R}^N}\left\vert \left(u(s^{-1 / N}x)\right)\right\vert^{2}dx=s\int_{\mathbb{R}^N}\vert u(x)\vert^{2}dx.$$
 For every $s \geq 1$ and $i=1,2$, since {$J_{\varepsilon}\left(u_i\right)=m_{\varepsilon}\left(c_i\right)$}, we have
		\begin{align*}
			m_{\varepsilon}\left(\sqrt{s} c_i\right) &\leq J_{\varepsilon}\left(u_i\left(s^{-1 / N}x\right)\right)\\
			&=s\left(\frac{1}{2 s^{2 / N}} \int_{\mathbb{R}^N}\left|\nabla u_i\right|^2 d x+\frac{1}{q s^{q / N}} \int_{\mathbb{R}^N}\left|\nabla u_i\right|^q d x-\int_{\mathbb{R}^N} G^{\varepsilon}\left(u_i\right) d x\right)\\
			&\leq s J_{\varepsilon}\left(u_i\right)\\
			&=s m_{\varepsilon}\left(c_i\right) .
		\end{align*}
		Moreover, if $s>1$ and $u_i \neq 0$, then $m_{\varepsilon}\left(\sqrt{s} c_i\right)<s m_{\varepsilon}\left(c_i\right)$ for $i=1, 2$. Without loss of generality, assume that $u_1 \neq 0$. If $c_1 \geq c_2$, then
		$$
		m_{\varepsilon}\left(\sqrt{c_1^2+c_2^2}\right)<\frac{c_1^2+c_2^2}{c_1^2} m_{\varepsilon}\left(c_1\right)=m_{\varepsilon}\left(c_1\right)+\frac{c_2^2}{c_1^2} m_{\varepsilon}\left(c_1\right) \leq m_{\varepsilon}\left(c_1\right)+m_{\varepsilon}\left(c_2\right) .
		$$
	If $c_1<c_2$, then
		$$
		m_{\varepsilon}\left(\sqrt{c_1^2+c_2^2}\right) \leq \frac{c_1^2+c_2^2}{c_2^2} m_{\varepsilon}\left(c_2\right)=\frac{c_1^2}{c_2^2} m_{\varepsilon}\left(c_2\right)+m_{\varepsilon}\left(c_2\right)<m_{\varepsilon}\left(c_1\right)+m_{\varepsilon}\left(c_2\right) .
		$$
The proof is completed.
	\end{proof}
	\begin{remark}\label{reJ}
{\rm		(a) Lemmas \ref{LemmaSubadd} and \ref{LemmaStrictSubadd}  still hold if $J_{\varepsilon}$ is replaced with $J$.\\
		{(b) The condition \( J_{\varepsilon}(u_i) = m_{\varepsilon}(c_i), i = 1, 2 \) can be relaxed to \( J_{\varepsilon}(u_1) = m_{\varepsilon}(c_1) \) or \( J_{\varepsilon}(u_2) = m_{\varepsilon}(c_2) \), and the
		Lemma \ref{LemmaStrictSubadd} still holds}.}
	\end{remark}
	
	\begin{Lemma}\label{LemmaPSseq}
		Assume that $c>0$ and $(g_0)-(g_3)$ hold. Let ${(\tilde{u}_n)_n} \subset \mathcal{D}(c)$ be a minimizing sequence for $J_{\varepsilon}$ at level $m_{\varepsilon}(c)$. Then, there exists another minimizing sequence ${(u_n)_n} \subset \mathcal{D}(c)$ bounded in $X$, and $\lambda_{\varepsilon} \in \mathbb{R}$ such that {for all $\varphi \in X$}
		$$
		\left\|u_n-\tilde{u}_n\right\|_X \rightarrow 0,
\quad J_{\varepsilon}^{\prime}\left(u_n\right) \varphi+\lambda_{\varepsilon}\int_{\mathbb{R}^N} u_n \varphi d x \rightarrow 0  \text { as } n \rightarrow\infty.
		$$
	Moreover, if $\lim \limits _{n \rightarrow \infty}\left\|u_n\right\|_2<c$, then $\lambda_{\varepsilon}=0$.
	\end{Lemma}
	
	\begin{proof}
		Let {$(\tilde{u}_n)_n$ be} a minimizing sequence for $J_{\varepsilon}$ at level $m_{\varepsilon}(c)$.
		{By} Ekeland's variational principle \cite[Theorem 2.4]{Willem1996Minimax}, we derive a new minimizing sequence ${(u_n)_n} \subset \mathcal{D}(c)$, that is also a Palais-Smale sequence for $J_{\varepsilon}$ on $\mathcal{D}(c)$. By \cite[Proposition 5.12]{Willem1996Minimax}, there exist ${(\lambda_n)_n}\subset\mathbb{R}$ , such that
		{for all $\varphi \in X$}
		$$
		\left\|u_n-\tilde{u}_n\right\|_X \rightarrow 0,\quad J_{\varepsilon}^{\prime}\left(u_n\right) \varphi+\lambda_n \int_{\mathbb{R}^N} u_n \varphi d x \rightarrow 0  \text { as } n \rightarrow\infty.
		$$	
		Now we prove that
	 $({u}_n)_n$ is bounded in $X$.
		If $\frac{2 N}{N+2} < q < 2$, from \eqref{Eq-GN1} and Lemma \ref{Lemmam(c)<0} we have that
			\begin{equation}\label{Eq-GN11}
			\begin{aligned}
				0>J_{\varepsilon}(u_n) \geq \frac{1}{2}\|\nabla u_n\|_2^2-\int_{\mathbb{R}^N} G_{+}(u_n) d x
				\geq\|\nabla u_n\|_2^2\left(\frac{1}{2}-\delta C_{N, \bar{q}}^{\bar{q}} c^{\bar{q}\left(1-\delta_{\bar{q}}\right)}\|\nabla u_n\|_2^{\bar{q} \delta_{\bar{q}}-2}\right)-C_\delta c^2 .
			\end{aligned}
		\end{equation}
Therefore, $\|\nabla u_n\|_2$ is bounded in $\mathbb{R}$,
from \eqref{eq-G} and Lemma \ref{LemmaEmbed}, we know that $ \int_{\mathbb{R}^N} G_{+}(u_n) d x  $ is bounded in $\mathbb{R}$. Since
$$
0>J_{\varepsilon}(u_n)  \geq \frac{1}{q}\|\nabla u_n\|_q^q+\frac{1}{2}\|\nabla u_n\|_2^2-\int_{\mathbb{R}^N} G_{+}(u_n) d x,
$$
it follows that $\left\|\nabla u_n\right\|_q$ is also bounded in $\mathbb{R}$. Thus, $({u}_n)_n$ is bounded in $X$. If $2< q <N$, The proof is similar, and we omit it here.
 Therefore, there exists ${u}_{\varepsilon}
{\in X}$ such that, {up to a subsequence,}
${u}_n \rightharpoonup {u}_{\varepsilon}$ in $X$ and ${u}_n \rightarrow {u}_{\varepsilon}$ in $L_{loc}^m\left(\mathbb{R}^N\right)$ for every {$m$,
with} $2 \leq m<q^{\prime}$ and ${u}_n \rightarrow {u}_{\varepsilon}$ for a.e. in $\mathbb{R}^N$. By Fatou's lemma, it follows that ${u}_{\varepsilon} \in \mathcal{D}(c)$.
 Let $\varphi=u_n$, it is easy to show that {$(\lambda_n)_n$} is bounded in $\mathbb{R}$.
 we may assume $\lambda_n\rightarrow \lambda_\varepsilon \text { as } n \rightarrow\infty $, up to a subsequence if necessary. Hence
		$$
	 J_{\varepsilon}^{\prime}\left(u_n\right) \varphi+\lambda_{\varepsilon}\int_{\mathbb{R}^N} u_n \varphi d x \rightarrow 0  \text { as } n \rightarrow\infty.
		$$
		If  $\lim \limits _{n \rightarrow \infty}\left\|u_n\right\|_2<c$, {then} ${u}_{\varepsilon} \in\mathcal{D}(c)\backslash\mathcal{S}(c) $ and is an interior point of $\mathcal{D}(c)$. {Therefore,} ${u}_{\varepsilon}$ is a local minimizer of $J$ on $X$. Hence
		$$
		J_{\varepsilon}^{\prime}\left({u}_{\varepsilon}\right) \varphi=0 \quad \mbox{for all }\varphi \in {X},
		$$
this implies that  $\lambda_\varepsilon=0$.
	\end{proof}
	
	{\begin{Lemma}\label{Lem3-ae}
	Assume that $c>0$ and $(g_0)-(g_3)$ hold. Let ${(\tilde{u}_n)_n} \subset \mathcal{D}(c)$ be a minimizing sequence for $J_{\varepsilon}$ at level $m_{\varepsilon}(c)$.
		Then, there exists another minimizing sequence $({u}_n)_n \subset \mathcal{D}(c)$ for $J_{\varepsilon}$, such that for some $u_{\varepsilon}\in X$,
		$$
		\nabla u_n \rightarrow \nabla u_{\varepsilon} \text { a.e. } \text{ in } \mathbb{R}^N.
		$$
	\end{Lemma}}
	\begin{proof}
		Define
		$$g^{\varepsilon}(s):=\frac{d}{d s} G^{\varepsilon}(s),\quad
		s \in \mathbb{R}.$$
		From Lemma \ref{LemmaPSseq}, we know that there exists anthor bounded sequence $ ({u}_n)_n $ such that for any $v \in {X}$,
		\begin{equation}\label{Eq-aePhi1}
			\begin{aligned}
				o_n(1)=&\left\langle J_{\varepsilon}^{\prime}\left(u_n\right), v \right\rangle+\lambda_{\varepsilon}\int_{\mathbb{R}^N}  u_n vdx\\
				=&\int_{\mathbb{R}^N} \left( \nabla u_n\nabla v
				+ | \nabla u_n|^{q-2} \nabla u_n\nabla v
				+ \lambda_{\varepsilon}u_n v\right)dx\\
				&-\int_{\mathbb{R}^N}g^{\varepsilon}(u_n)v
				dx.\\
			\end{aligned}	
		\end{equation}
	 Up to a subsequence, we may assume that
	${u}_n \rightharpoonup {u}_{\varepsilon}$ in $X$.
	Therefore, for any $v \in {X}$,
		\begin{equation}\label{Eq-aePhi2}
			\left\langle J_{\varepsilon}^{\prime}\left(u_{\varepsilon}\right),
			v\right\rangle
			+\lambda_{\varepsilon}\int_{\mathbb{R}^N}u_{\varepsilon}vdx
			=\lim_{n \rightarrow \infty}\left( \left\langle J_{\varepsilon}^{\prime}\left(u_n\right),
			v\right\rangle
			+\lambda_{\varepsilon}\int_{\mathbb{R}^N}u_nvdx \right)
			=0.
		\end{equation}
		Now we use a technique due to Boccardo and Murat \cite{Boccardo1992}. Fix $k\in \mathbb{R}^{+}$, define the function
		$$
		\tau_k(s)= \begin{cases}s & \text { if }|s| \leq k, \\ k s /|s| & \text { if }|s|>k.\end{cases}
		$$
		It's easy to see that $(\tau_k\left(u_n-u_{\varepsilon}\right))_{n} $ is bounded in $X$. Fix a function $\psi \in C_0^{\infty}\left(\mathbb{R}^N\right)$ with $0 \leq \psi \leq 1$ in $\mathbb{R}^N$, $\psi(x)=1$ for $x \in B_{1}(0)$ and $\psi(x)=0$ for $x \in \mathbb{R}^N \backslash B_2(0)$.\
		Now, take $R>0$ and define $\psi_R(x)=\psi(x / R)$ for $x \in \mathbb{R}^N$.
		We obtain from
		\eqref{Eq-aePhi1} and \eqref{Eq-aePhi2} that
		\begin{equation}\label{Eq-aePhi3}
			\begin{aligned}
				o_n(1)=&\left\langle J_{\varepsilon}^{\prime}\left(u_n\right), \tau_k(u_n-u_{\varepsilon})\psi_R\right\rangle
				+\lambda_{\varepsilon}\int_{\mathbb{R}^N}u_n\tau_k(u_n-u_{\varepsilon})\psi_Rdx\\
				=& \left\langle J_{\varepsilon}^{\prime}\left(u_n\right)-J_{\varepsilon}^{\prime}(u_{\varepsilon}), \tau_k(u_n-u_{\varepsilon})\psi_R\right\rangle
				+\lambda_{\varepsilon}\int_{\mathbb{R}^N} \left( u_n- u_{\varepsilon} \right) \tau_k(u_n-u_{\varepsilon})\psi_Rdx \\
				=&\int_{\mathbb{R}^N} \left(\left|\nabla u_n\right|^{q-2} \nabla u_n-|\nabla u_{\varepsilon}|^{q-2}\nabla u_{\varepsilon} \right)\nabla \left(\tau_k(u_n-u_{\varepsilon})\psi_R\right)dx\\
				&+\int_{\mathbb{R}^N} \left(\nabla u_n-\nabla u_{\varepsilon} \right)\nabla \left(\tau_k(u_n-u_{\varepsilon})\psi_R\right)dx\\
				&+\lambda_{\varepsilon}\int_{\mathbb{R}^N} \left( u_n- u_{\varepsilon} \right) \tau_k(u_n-u_{\varepsilon})\psi_Rdx\\
				&-\int_{\mathbb{R}^N} \left(g^{\varepsilon}(u_n)-g^{\varepsilon}(u_{\varepsilon}) \right) \tau_k(u_n-u_{\varepsilon})\psi_Rdx\\
			\end{aligned}
		\end{equation}
	and
		\begin{equation}\label{Eq-aePhi3en}
		\begin{aligned}
			o_n(1)=&\left\langle J_{\varepsilon}^{\prime}\left(u_n\right), (u_n-u_{\varepsilon})\psi_R\right\rangle
			+\lambda_{\varepsilon}\int_{\mathbb{R}^N}u_n(u_n-u_{\varepsilon})\psi_Rdx\\
			=& \left\langle J_{\varepsilon}^{\prime}\left(u_n\right)-J_{\varepsilon}^{\prime}(u_{\varepsilon}), (u_n-u_{\varepsilon})\psi_R\right\rangle
			+\lambda_{\varepsilon}\int_{\mathbb{R}^N} \left( u_n- u_{\varepsilon} \right) (u_n-u_{\varepsilon})\psi_Rdx \\
			=&\int_{\mathbb{R}^N} \left(\left|\nabla u_n\right|^{q-2} \nabla u_n-|\nabla u_{\varepsilon}|^{q-2}\nabla u_{\varepsilon} \right)\nabla \left((u_n-u_{\varepsilon})\psi_R\right)dx\\
			&+\int_{\mathbb{R}^N} \left(\nabla u_n-\nabla u_{\varepsilon} \right)\nabla \left((u_n-u_{\varepsilon})\psi_R\right)dx\\
			&+\lambda_{\varepsilon}\int_{\mathbb{R}^N} \left( u_n- u_{\varepsilon} \right)^2\psi_Rdx\\
			&-\int_{\mathbb{R}^N} \left(g^{\varepsilon}(u_n)-g^{\varepsilon}(u_{\varepsilon}) \right) (u_n-u_{\varepsilon})\psi_Rdx.\\
		\end{aligned}
	\end{equation}
		Since  $ ({u}_n)_n $  is bounded in $X$, up to a subsequence, we have
		\begin{equation}\label{Eq-aePhi4}
			\begin{aligned}
				\int_{\mathbb{R}^N} \left( u_n- u_{\varepsilon} \right) \tau_k(u_n-u_{\varepsilon})\psi_Rdx=o_n(1) \quad \text{and} \quad		\int_{\mathbb{R}^N} \left( u_n- u_{\varepsilon} \right)^2\psi_Rdx=o_n(1)
			\end{aligned}.
		\end{equation}
			From $(g_1)$ and $(g_3)$, when $N\geq3$, there exists $C>0$ such that
		\begin{equation}\label{eq-g1}
			|g^{\varepsilon}(s)| \leq C_1\left(|s|+|s|^{{q}^{\prime}-1}\right) \text{for all } s \in \mathbb{R}.
		\end{equation}
	Therefore, from \eqref{eq-g1} and Lemma \ref{LemmaEmbed}, we have
	\begin{equation}\label{eq-g2}
		\int_{\mathbb{R}^N}|g^{\varepsilon}(u_n)u_n|dx\leq C_1\int_{\mathbb{R}^N}\left(|u_n|^2+|u_n|^{{q}^{\prime}}\right)dx\leq C^{\prime}_1.
	\end{equation}
	When $N=2$,  there exists $C_2>0$ such that, for all $\alpha>0$
	\begin{equation}\label{eqq-g2}
	|g^{\varepsilon}(s)|\leq C_2\left(|s|+|s|\left(e^{\alpha s^2}-1\right)\right) \text { for all } s \in \mathbb{R}.
	\end{equation}
From $(g_1)$ and $(g_3)$, for any $\delta > 0$, there exists a constant $C_\delta > 0$ such that for every $s \in \mathbb{R}$,
\[
G_{+}(s) \leq \delta |s|^2 + C_\delta |s|^{\bar{q}}.
\]
Therefore, there exists $C_0 > 0$ such that $\|\nabla u_n\|_2 > C_0$. Otherwise, from \eqref{Eq-GN inequality1}, up to a subsequence, we have
\[
\lim_{n \to +\infty} \int_{\mathbb{R}^N} G_{+}(u_n) \, dx = 0.
\]
{This implies that $\lim_{n \rightarrow +\infty}J_{\varepsilon}(u_n)\geq0$, which contradicts $m_\varepsilon(c) < 0$}. Hence, there exists a constant $M > 0$ such that
\[
\left\|\frac{u_n}{\|\nabla u_n\|_2}\right\|_2 = \frac{\|u_n\|_2}{\|\nabla u_n\|_2} \leq M.
\]
If $  \|\nabla u_n\|_2 < \sqrt{2\pi / \alpha}  $,
	using Lemma \ref{LemmaTrudinger-Moser}, we have
	\begin{equation}\label{eqq-gc}
		\begin{aligned}
			\int_{\mathbb{R}^2}\left(e^{2 \alpha u_n^2}-1\right) {d} x & =\int_{\mathbb{R}^2}\left(e^{2 \alpha\|\nabla u_n\|_2^2\left(u_n /\|\nabla u_n\|_2\right)^2}-1\right) {d} x
			\leq C_\alpha .
		\end{aligned}
	\end{equation}
Since $ (u_n)_n $ is bounded in $X$, we may choose
 $\alpha>0$ small enough such that $\|u_n\|_X\leq\sqrt{\pi / \alpha}$. By \eqref{eqq-g2}, \eqref{eqq-gc} and Lemma \ref{LemmaEmbed},
 	\begin{equation}\label{eqq-g22}
 	\begin{aligned}
 	\int_{\mathbb{R}^2}|g^{\varepsilon}(u_n)u_n|dx&\leq C_2\int_{\mathbb{R}^2}\left(|u_n|^2+|u_n|^2\left(e^{\alpha |u_n|^2}-1\right)\right)dx\\
 	&\leq C_2\int_{\mathbb{R}^2}|u_n|^2dx+C_2\left(\int_{\mathbb{R}^2}\left(e^{2\alpha |u_n|^2}-1\right)dx\right)^{1/2}\left(\int_{\mathbb{R}^2}|u_n|^4dx\right)^{1/2}\\
 	&\leq C_2\int_{\mathbb{R}^2}|u_n|^2dx+C_2C^{1/2}_\alpha\left(\int_{\mathbb{R}^2}|u_n|^4dx\right)^{1/2}\\
 	&\leq C_2^{\prime}
 	.
 	\end{aligned}	
 	\end{equation}
Combining \eqref{eq-g2} and \eqref{eqq-g22}, we can deduce that $\int_{\mathbb{R}^N}|g^{\varepsilon}(u_n)u_n|dx\leq \max\{C^{\prime}_1,C^{\prime}_2\}$. Similarly, we can prove, there exists $C^{\prime}>0$ such that
{$$
\int_{\mathbb{R}^N}|g^{\varepsilon}(u_{\varepsilon})u_{\varepsilon}|dx\leq C^{\prime},\,
\int_{\mathbb{R}^N}|g^{\varepsilon}(u_{\varepsilon})u_n|dx\leq C^{\prime} \quad \text{ and } \quad
\int_{\mathbb{R}^N}|g^{\varepsilon}(u_n)u_{\varepsilon}|dx\leq C^{\prime}.
$$}
Hence, there exists $C>0$ such that
		\begin{equation}\label{Eq-aePhi5}
			\begin{aligned}
				&\quad\int_{\mathbb{R}^N} |\left(g^{\varepsilon}(u_n)- g^{\varepsilon}(u_{\varepsilon})\right) (u_n-u_{\varepsilon})\psi_R|dx\\
				&\leq\int_{\mathbb{R}^N} |\left(g^{\varepsilon}(u_n)- g^{\varepsilon}(u_{\varepsilon})\right) (u_n-u_{\varepsilon})|dx\\
				&\leq\int_{\mathbb{R}^N} |g^{\varepsilon}(u_{\varepsilon})u_{\varepsilon}| +|g^{\varepsilon}(u_n)u_n|+|g^{\varepsilon}(u_n)u_{\varepsilon}|+|g^{\varepsilon}(u_{\varepsilon})u_n|dx\\
				&\leq C.	
			\end{aligned}
		\end{equation}
	From \eqref{eq-g1} and \eqref{eqq-g2}, there exists a constant $C$ large enough such that
		{\begin{equation}\label{Eq-aePhi5ekn}
		\begin{aligned}
			&\quad\int_{\mathbb{R}^N} |\left(g^{\varepsilon}(u_n)- g^{\varepsilon}(u_{\varepsilon})\right) \tau_k(u_n-u_{\varepsilon})\psi_R|dx\\
			&\leq k	\int_{\mathbb{R}^N} |\left(g^{\varepsilon}(u_n)- g^{\varepsilon}(u_{\varepsilon})\right) \psi_R|dx\\
			&\leq k	\int_{B_{2R}(0)} |g^{\varepsilon}(u_n)- g^{\varepsilon}(u_{\varepsilon})|dx\\
			&\leq Ck.	
		\end{aligned}
	\end{equation}}
	Therefore,
		let
			\begin{equation*}
			\begin{aligned}
				e_n(x)=& \left(\left|\nabla u_n\right|^{q-2} \nabla u_n-|\nabla u_{\varepsilon}|^{q-2}\nabla u_{\varepsilon} \right)\nabla \left((u_n-u_{\varepsilon})\psi_R\right)\\&+
				\left(\nabla u_n-\nabla u_{\varepsilon} \right)\nabla \left((u_n-u_{\varepsilon})\psi_R\right)
			\end{aligned}
		\end{equation*}
	and
		\begin{equation*}
			\begin{aligned}
				e_{k,n}(x)=& \left(\left|\nabla u_n\right|^{q-2} \nabla u_n-|\nabla u_{\varepsilon}|^{q-2}\nabla u_{\varepsilon} \right)\nabla \left(\tau_k(u_n-u_{\varepsilon})\psi_R\right)\\&+
				\left(\nabla u_n-\nabla u_{\varepsilon} \right)\nabla \left(\tau_k(u_n-u_{\varepsilon})\psi_R\right).
			\end{aligned}
		\end{equation*}
	First, we give some estimates for
	$$
	\int_{\mathbb{R}^N} e_n(x) d x \quad \text { and } \quad \int_{\mathbb{R}^N} e_{k, n}(x) d x .
	$$
	From \eqref{Eq-aePhi3en} and \eqref{Eq-aePhi5}, we have
	\begin{equation}\label{enR}
		\int_{\mathbb{R}^N}e_{n}(x)dx\leq C+o_n(1).
	\end{equation}
	And from \eqref{Eq-aePhi3} and \eqref{Eq-aePhi5ekn}, we have
	\begin{equation}\label{eknR}
		\int_{\mathbb{R}^N}e_{k,n}(x)dx\leq Ck+o_n(1).
	\end{equation}
   Next, we give some estimates for
    $$
   \int_{B_{2 R}(0) \backslash B_R(0)} e_n(x) d x \quad \text { and } \quad \int_{B_{2 R}(0) \backslash B_R(0)} e_{k, n}(x) d x.
   $$
	{We may assume that there exist $C_R>0$ such that $|\nabla\psi_R|<C_R$. Then
	\begin{equation*}
		\begin{aligned}
			&\quad\left|\int_{B_{2R}(0) \backslash B_R(0)}\left|\nabla u_n\right|^{q-2} \nabla u_n\left((u_n-u_{\varepsilon})\nabla\psi_R\right)dx\right|\\
			&\leq
			\int_{B_{2R}(0)\backslash B_R(0)}\left|\nabla u_n\right|^{q-1} |u_n-u_{\varepsilon}||\nabla\psi_R|dx\\
			&\leq C_R\int_{B_{2R}(0) \backslash B_R(0)}\left|\nabla u_n\right|^{q-1}\left|u_n-u_{\varepsilon}\right|dx\\
			&\leq C_R \left(\int_{B_{2R}(0) \backslash B_R(0)}\left|\nabla u_n\right|^{q}dx \right)^{\frac{q-1}{q}} \left(  \int_{B_{2R}(0) \backslash B_R(0)}\left|u_n-u_{\varepsilon}\right|^{q}dx\right)^{\frac{1}{q}} \\
			&=o_n(1).
		\end{aligned}
	\end{equation*}
	Similarly, we can prove
	$$
	\left|\int_{B_{2R}(0) \backslash B_R(0)}          \left|\nabla u_{\varepsilon}\right|^{q-2}    \nabla u_{\varepsilon}\left((u_n-u_{\varepsilon})\nabla\psi_R\right)dx\right|=o_n(1),
	$$
	$$
	\left|\int_{B_{2R}(0) \backslash B_R(0)}          \nabla u_n\left((u_n-u_{\varepsilon})\nabla\psi_R\right)dx\right|=o_n(1),
	$$
	$$
	\left|\int_{B_{2R}(0) \backslash B_R(0)}          \nabla u_{\varepsilon}\left((u_n-u_{\varepsilon})\nabla\psi_R\right)dx\right|=o_n(1).
	$$
	Therefore
			\begin{equation}\label{Eq-eue}
			\begin{aligned}
			\int_{B_{2R}(0) \backslash B_R(0)}e_{n}(x)dx
			=&\int_{B_{2R}(0) \backslash B_R(0)} \left(\left|\nabla u_n\right|^{q-2} \nabla u_n-|\nabla u_{\varepsilon}|^{q-2}\nabla u_{\varepsilon} \right)\nabla \left((u_n-u_{\varepsilon})\psi_R\right)dx\\
			&+\int_{B_{2R}(0) \backslash B_R(0)}\left(\nabla u_n-\nabla u_{\varepsilon} \right)\nabla \left((u_n-u_{\varepsilon})\psi_R\right)dx\\
			=&\int_{B_{2R}(0) \backslash B_R(0)} \left(\left|\nabla u_n\right|^{q-2} \nabla u_n-|\nabla u_{\varepsilon}|^{q-2}\nabla u_{\varepsilon} \right)(\nabla u_n-\nabla u_{\varepsilon})\psi_Rdx\\
			&+\int_{B_{2R}(0) \backslash B_R(0)}\left(\nabla u_n-\nabla u_{\varepsilon} \right) (\nabla u_n-\nabla u_{\varepsilon})\psi_Rdx\\
			&+\int_{B_{2R}(0) \backslash B_R(0)} \left(\left|\nabla u_n\right|^{q-2} \nabla u_n-|\nabla u_{\varepsilon}|^{q-2}\nabla u_{\varepsilon} \right) \left((u_n-u_{\varepsilon})\nabla\psi_R\right)dx\\
			&+\int_{B_{2R}(0) \backslash B_R(0)}\left(\nabla u_n-\nabla u_{\varepsilon} \right) \left((u_n-u_{\varepsilon})\nabla\psi_R\right)dx\\
			\geq&
			\int_{B_{2R}(0) \backslash B_R(0)} \left(\left|\nabla u_n\right|^{q-2} \nabla u_n-|\nabla u_{\varepsilon}|^{q-2}\nabla u_{\varepsilon} \right) \left((u_n-u_{\varepsilon})\nabla\psi_R\right)dx\\
			&+\int_{B_{2R}(0) \backslash B_R(0)}\left(\nabla u_n-\nabla u_{\varepsilon} \right) \left((u_n-u_{\varepsilon})\nabla\psi_R\right)dx\\
			=&o_n(1).
			\end{aligned}
		\end{equation}}
	Hence
	\begin{equation}\label{enB21}
	\int_{B_{2R}(0) \backslash B_R(0)}e_{n}(x)dx\geq o_n(1).
	\end{equation}	
	Using the same proof, we obtain
	\begin{equation}\label{eknB21}
	\int_{B_{2R}(0) \backslash B_R(0)}e_{k,n}(x)dx\geq o_n(1).
	\end{equation}	
		{Finally, we give some estimates for
		$$
		\int_{B_{R}(0) } e_n(x) d x \quad \text { and } \quad \int_{B_{R}(0)} e_{k, n}(x) d x.
		$$
		Combining \eqref{enR} and \eqref{enB21}, we obtain that
		\begin{equation}\label{Eq-aePhi6en}
			\begin{aligned}
				\int_{B_R(0)}e_{n}(x)dx
				=&\int_{\mathbb{R}^N}e_{n}(x)dx-\int_{B_{2R}(0) \backslash B_R(0)}e_{n}(x)dx\\
				\leq& C+o_n(1).
			\end{aligned}
		\end{equation}
		Combining \eqref{eknR} and \eqref{eknB21}, we obtain that
		\begin{equation}\label{Eq-aePhi6}
		\begin{aligned}
		\int_{B_R(0)}e_{k,n}(x)dx
		=&\int_{\mathbb{R}^N}e_{k,n}(x)dx-\int_{B_{2R}(0) \backslash B_R(0)}e_{k,n}(x)dx\\
		\leq& Ck+o_n(1).
		\end{aligned}
		\end{equation}}
		Take $0<\theta<1$ and split $B_R(0)$ into
		$$
		S_n^k=\left\{x \in B_R(0)	|\enspace| u_n-u_{\varepsilon} | \leq k\right\}, \quad G_n^k=\left\{x \in B_R(0)| \enspace| u_n-u_{\varepsilon} |>k\right\} .
		$$
			By Lemma \ref{Remarkpost}, $e_n(x) \geq 0$ and $e_{k,n}(x) \geq 0$ in $B_R(0)$, therefore
		\begin{equation}
			\begin{aligned}
				\int_{B_R(0)} e_n^\theta d x & =
				\int_{S_n^k} e_n^\theta d x+\int_{G_n^k} e_n^\theta d x \\
				& \leq\left(\int_{S_n^k} e_n d x\right)^\theta\left|S_n^k\right|^{1-\theta}+\left(\int_{G_n^k} e_n d x\right)^\theta\left|G_n^k\right|^{1-\theta}\\
				& =\left(\int_{S_n^k} e_{k,n} d x\right)^\theta\left|S_n^k\right|^{1-\theta}+\left(\int_{G_n^k} e_n d x\right)^\theta\left|G_n^k\right|^{1-\theta}.
			\end{aligned}
		\end{equation}
		For fixed $k\in \mathbb{R}^{+}$, $\left|G_n^k\right| \rightarrow 0$ as $n \rightarrow \infty$, and from \eqref{Eq-aePhi6en} and \eqref{Eq-aePhi6}, we get
		\begin{equation}
			\begin{aligned}
				\int_{B_R(0)} e_n^\theta d x
				& \leq\left(\int_{S_n^k} e_{k,n} d x\right)^\theta\left|S_n^k\right|^{1-\theta}+\left(\int_{G_n^k} e_n d x\right)^\theta\left|G_n^k\right|^{1-\theta}\\
				& \leq\left(\int_{S_n^k} e_{k,n} d x\right)^\theta\left|S_n^k\right|^{1-\theta}+o_n(1)\\
				&  \leq(C k)^\theta|B_R(0)|^{1-\theta}+o_n(1).
			\end{aligned}
		\end{equation}
		Let $k\rightarrow 0$, we get that $e_n^\theta \rightarrow 0$ in $L^1(B_R(0))$ as $n\rightarrow \infty$. By Lemma \ref{LemmaConverae}, we have
		$$
		\nabla u_n \rightarrow \nabla u_{\varepsilon} \text { a.e. } \text{ in } B_R(0).
		$$
		Since $R$ is arbitrary,
		by passing to a subsequence, we have
		$$
		\nabla u_n \rightarrow \nabla u_{\varepsilon} \text { a.e. } \text{ in } \mathbb{R}^N.
		$$
	\end{proof}
	\begin{Lemma}\label{LemmaStrongconver}
		If $g$ satisfy $(g_0)-(g_4)$, then, for every $\beta>0$, there exists $\bar{c}\geq0$ such that $m_{\varepsilon}(c)<-\beta$ for every $c>\bar{c}$ and {$\varepsilon>0$}. If $0<c_n \rightarrow c$ and $\tilde{u}_n \in \mathcal{D}\left(c_n\right)$ {is} such that $J_{\varepsilon}\left(\tilde{u}_n\right) \rightarrow m_{\varepsilon}(c)$, then there exists  another minimizing sequence ${(u_n)_n \subset} \mathcal{D}(c)$ such that $\left\|u_n-\tilde{u}_n\right\|_X \rightarrow 0$ and up to translations, $u_n \rightarrow u$ in $L^m\left(\mathbb{R}^N\right)$ for {all $m$, with} $2<m<q^{\prime}$.
	\end{Lemma}
	\begin{proof}
		Let $\bar{c}$ be {the number}  determined by Lemma \ref{Lemmam(c)<0} and ${(\tilde{u}_n)_n} \subset \mathcal{D}(c)$ be a minimizing sequence of $m_{\varepsilon}(c)$.
{Hence, Lemma~\ref{LemmaPSseq}
gives that} $m_{\varepsilon}(c)$ possesses another minimizing sequence ${(u_n)_n} \subset \mathcal{D}(c)$ and $\lambda_{\varepsilon} \in \mathbb{R}$ such that
{for all $\varphi \in X^{*}$}
		$$
		\left\|u_n-\tilde{u}_n\right\|_X \rightarrow 0, \quad J_{\varepsilon}^{\prime}\left(u_n\right) \varphi+\lambda_{\varepsilon}\int_{\mathbb{R}^N} u_n \varphi d x \rightarrow 0  \text { as } n \rightarrow\infty ,
		$$
and {$(u_n)_n$} is also a Palais-Smale sequence for $J_{\varepsilon}$ on $\mathcal{D}(c)$.
		Similarly to the proof of Lemma \ref{LemmaBoundedbelow}, we have that {$(u_n)_n$} is bounded in $X$. If
		$$
		\lim _{n \rightarrow \infty} \sup _{y \in \mathbb{R}^N} \int_{B_R(y)}\left|u_n(x)\right|^2 d x=0
		$$
		for any $R>0$, due to Lemma \ref{LemmaLions},  $\left\|u_n\right\|_m \rightarrow 0$ for any
{$m$, with}
$2<m<q^{\prime}$. From $(g_1)$ and $(g_3)$, for any $\delta>0$, there exists $C_\delta>0$ such that for every $s \in \mathbb{R}$,
		$$
		G_{+}(s) \leq \delta|s|^2+C_\delta|s|^{\bar{q}}.
		$$
{Thus,} $\lim \limits_{n \rightarrow \infty} \int_{\mathbb{R}^N} G_{+}\left(u_n\right) d x=0$ and $m_{\varepsilon}(c)=\lim \limits_{n \rightarrow\infty} J_{\varepsilon}\left(u_n\right) \geq 0$, which  { contradicts} Lemma~\ref{Lemmam(c)<0}.
		Hence, there exist   $\varepsilon_0>0$ and a sequence $(y_n)_n \subset \mathbb{R}^N$ such that, {for sufficiently large $R>0$}
		$$
		\int_{B_R\left(y_n\right)}\left|u_n(x)\right|^2 d x \geq \varepsilon_0>0.
		$$
		 Moreover, we have $u_n\left(x+y_n\right) \rightharpoonup u_c \not \equiv 0$ in $X$ for some {$u_c \in X$}.

{Put} $v_n(x):=u_n\left(x+y_n\right)-u_c(x)$.
{Then,} $v_n \rightharpoonup 0$ in $X$,  and  $u_n\left(x+y_n\right)\rightarrow u_c$ for a.e. $x\in \mathbb{R}^N$ by Lemma~\ref{LemmaEmbed}. Therefore, we obtain
	\begin{equation}
		\begin{aligned}
			& \left\|\nabla u_n\right\|_2^2=\left\|\nabla u_n\left(\cdot+y_n\right)\right\|_2^2=\left\|\nabla v_n\right\|_2^2+\left\|\nabla u_c\right\|_2^2+o_n(1).
		\end{aligned}
	\end{equation}
	 Moreover, from the Br{\'e}zis-Lieb lemma \cite[Lemma 3.2]{Jeanjean2019Nonradial}, we have
	 \begin{equation}
	 	\begin{aligned}
	 		& \left\|u_n\right\|_2^2=\left\|u_n\left(\cdot+y_n\right)\right\|_2^2=\left\|v_n\right\|_2^2+\left\|u_c\right\|_2^2+o_n(1), \\
	 		\int_{\mathbb{R}^N}g^{\varepsilon}\left(u_n\right)dx=&	\int_{\mathbb{R}^N}g^{\varepsilon}\left(u_n\left(\cdot+y_n\right)\right)dx=\int_{\mathbb{R}^N}g^{\varepsilon}\left(v_n\right)dx+\int_{\mathbb{R}^N}g^{\varepsilon}(u_c)dx+o_n(1) .
	 	\end{aligned}
	 \end{equation}
	{Now, we prove that}
	\begin{equation}\label{BL1}
		\left\|\nabla u_n\right\|_q^q=\left\|\nabla u_n\left(\cdot+y_n\right)\right\|_q^q=\left\|\nabla v_n\right\|_q^q+\left\|\nabla u_c\right\|_q^q+o_n(1) .
	\end{equation}
		We next claim that
		$$
		\lim_{n \rightarrow\infty} \sup _{y \in \mathbb{R}^N} \int_{B_1(y)}\left|v_n\right|^2 d x=0,
		$$
		which, {by Lemma \ref{LemmaLions}, will yield the statement of Lemma \ref{LemmaConverae}}. If this is not true, then, as before, there exist $z_n \in \mathbb{R}^N$ and $v_c \in X \backslash\{0\}$ such that, denoting $w_n(x):=v_n\left(x-z_n\right)-v_c(x)$, we have $w_n \rightharpoonup 0$ in $X, w_n \rightarrow 0$ for a.e. $x\in \mathbb{R}^N$, and
		$$
		\lim_{n \rightarrow\infty}
\left(J_{\varepsilon}\left(v_n\right)-J_{\varepsilon}\left(w_n\right)
\right)=J_{\varepsilon}(v) .
		$$
Note that, once more due to the Br{\'e}zis-Lieb lemma,
		{$$
		\lim_{n \rightarrow\infty}\left( \left\|u_n\right\|_2^2-\left\|w_n\right\|_2^2\right)=\lim_{n \rightarrow\infty}\left(\left\|u_n\right\|_2^2-\left\|v_n\right\|_2^2+\left\|v_n\right\|_2^2-\left\|w_n\right\|_2^2\right)=\|u_c\|_2^2+\|v_c\|_2^2,
		$$
}
		whence, denoting $\beta:=\|u_c\|_2>0$ and $\gamma:=\|v_c\|_2>0$, there holds
		$$
		c^2-\beta^2-\gamma^2 \geq \liminf _n\left\|u_n\right\|_2^2-\beta^2-\gamma^2=\liminf _n\left\|w_n\right\|_2^2=: \delta^2 \geq 0 .
		$$
If $\delta>0$, then let us set $\tilde{w}_n:=\frac{\delta}{\left\|w_n\right\|_2} w_n \in \mathcal{S}(\delta)$. {Via explicit computations}, we have
$$\lim \limits_{n \rightarrow\infty} \big[J_{\varepsilon}\left(w_n\right)
-J_{\varepsilon}\left(\tilde{w}_n\right)\big]=0.$$
Hence, together with Lemma \ref{LemmaSubadd} and since $m_{\varepsilon}(c)$ is non-increasing with respect to $c>0$,
		\begin{equation}\label{eqSub}
			\begin{aligned}
				m_{\varepsilon}(c) & =\lim \limits_{n \rightarrow\infty} J_{\varepsilon}\left(u_n\right)\\
				&=J_{\varepsilon}(u_c)+J_{\varepsilon}(v_c)+\lim_{n \rightarrow\infty} J_{\varepsilon}\left(w_n\right)\\
				&=J_{\varepsilon}(u_c)+J_{\varepsilon}(v_c)+\lim_{n \rightarrow\infty} J_{\varepsilon}\left(\tilde{w}_n\right) \\
				& \geq m_{\varepsilon}(\beta)+m_{\varepsilon}(\gamma)+m_{\varepsilon}(\delta)\\
				& \geq m_{\varepsilon}\left(\sqrt{\beta^2+\gamma^2+\delta^2}\right) \\
				&\geq m_{\varepsilon}(c) .
			\end{aligned}
		\end{equation}
Thus all the inequalities in \eqref{eqSub} are in fact equalities and, in particular, $J_{\varepsilon}(u_c)=m_{\varepsilon}(\beta)$ and $J_{\varepsilon}(v_c)=$ $m_{\varepsilon}(\gamma)$. Therefore Lemma \ref{LemmaStrictSubadd} yields {that} $m_{\varepsilon}(\beta)+m_{\varepsilon}(\gamma)+m_{\varepsilon}(\delta)>m_{\varepsilon}\left(\sqrt{\beta^2+\gamma^2+\delta^2}\right)$, which contradicts \eqref{eqSub}.

If $\delta=0$, then $w_n \rightarrow 0$ in $L^2\left(\mathbb{R}^N\right)$, which, together with Lemma \ref{LemmaGN}, implies {that}
$$\lim \limits_{n \rightarrow\infty} \int_{\mathbb{R}^N} G_{+}\left(w_n\right) d x=0,$$
whence $\lim \limits \inf _n J_{\varepsilon}\left(w_n\right) \geq 0$. Then \eqref{eqSub} becomes
		$$
		m_{\varepsilon}(c)=\lim_{n \rightarrow\infty} J_{\varepsilon}\left(u_n\right)=J_{\varepsilon}(u_c)+J_{\varepsilon}(v_c)+\lim_{n \rightarrow\infty} J_{\varepsilon}\left(w_n\right) \geq m_{\varepsilon}(\beta)+m_{\varepsilon}(\gamma) \geq m_{\varepsilon}\left(\sqrt{\beta^2+\gamma^2}\right)=m_{\varepsilon}(c)
		$$
		and we get a contradiction as before.
	\end{proof}
\medskip

\noindent	
	\textbf{Proof of Theorem \ref{theoremPerturbed}.}
	Let $\bar{c}\geq0$ be determined by Lemma \ref{Lemmam(c)<0} and {let  $(u_n )_n\subset \mathcal{D}(c)$ be
a minimizing sequence} for $m_{\varepsilon}(c)$. Then, in virtue of Lemmas \ref{LemmaBoundedbelow} and \ref{LemmaStrongconver}, there exists $u_{\varepsilon} \in \mathcal{D}(c) \backslash 0$ such that, up to subsequences and translations, $u_n \rightharpoonup u_{\varepsilon}$ in $X$ and $u_n \rightarrow u_{\varepsilon}$ in $L^m\left(\mathbb{R}^N\right)$ for every
{$m$, with}
$2<m<q^{\prime}$ and  $u_n \rightarrow u_{\varepsilon}$ for a.e. $x\in \mathbb{R}^N$. From $(g_1)$ and $(g_3)$ we have ${\lim \limits_{n \rightarrow\infty}} \int_{\mathbb{R}^N} G_{+}\left(u_n\right) d x=\int_{\mathbb{R}^N} G_{+}\left(u_{\varepsilon}\right) d x$. {{Therefore}}, using Fatou's lemma, $m_{\varepsilon}(c) \leq J_{\varepsilon}\left(u_{\varepsilon}\right) \leq \lim \limits_{n \rightarrow\infty} J_{\varepsilon}\left(u_n\right)=m_{\varepsilon}(c)<-\beta$. In particular, there exists $\lambda_{\varepsilon} \in \mathbb{R}$ such that
	$$
	-\Delta u_{\varepsilon}-\Delta_q u_{\varepsilon}+\lambda_{\varepsilon} u_{\varepsilon}=g^{\varepsilon}\left(u_{\varepsilon}\right), \quad x\in \mathbb{R}^N.
	$$
	By Lemma \ref{LemmaPSseq},
	note that $\lambda_{\varepsilon}=0$ if $\lim \limits _{n \rightarrow \infty}\left\|u_n\right\|_2<c$.
	From \cite[Lemma 2.3]{Baldelli20222q}, we know that if \(u_{\varepsilon}\) is a solution of equation \eqref{Equationperturbed}, then \(u_{\varepsilon}\) satisfies the following Pohozaev identity:
	\begin{equation}\label{eqPoh}
		\frac{N-2}{2} \int_{\mathbb{R}^N}|\nabla u_{\varepsilon}|^2 \, dx + \frac{N-q}{q} \int_{\mathbb{R}^N}|\nabla u_{\varepsilon}|^q \, dx + \frac{\lambda N}{2} \int_{\mathbb{R}^N}|u_{\varepsilon}|^2 \, dx = N \int_{\mathbb{R}^N} G^{\varepsilon}(u_{\varepsilon}) \, dx.
	\end{equation}
	If $\lambda_{\varepsilon} \leq 0$,  \eqref{eqPoh} yields $$
	0>J_{\varepsilon}\left(u_{\varepsilon}\right) \geq\left( \left\|\nabla u_{\varepsilon}\right\|_2^2+\left\|\nabla u_{\varepsilon}\right\|_q^q\right) / N \geq 0,
	$$ which is a contradiction. Therefore $\lambda_{\varepsilon}>0$ and $\lim \limits _{n \rightarrow \infty}\left\|u_n\right\|_2=c$.

	If $\|u_{\varepsilon}\|_2=c_0<c$,  then arguing as in Lemma \ref{LemmaStrongconver}, we derive {that}
	$u_n \rightarrow u_{\varepsilon} $ a.e. in $\mathbb{R}^N$
	and
	$\nabla u_n \rightarrow \nabla u_{\varepsilon} $ a.e. in $ \mathbb{R}^N$.
	Thus,
	$$
	m_{\varepsilon}(c)=\lim_{n \rightarrow\infty} J_{\varepsilon}\left(u_n\right)=J_{\varepsilon}(u_c)+\lim_{n \rightarrow\infty}J_{\varepsilon}(u_n-u_c) \geq m_{\varepsilon}(c_0)+m_{\varepsilon}(c-c_0) \geq
	m_{\varepsilon}(c).
	$$
	And we reach a contradiction as before.
	Therefore $\lim \limits_{n \rightarrow\infty}\left\|u_n\right\|_2
=\left\|u_{\varepsilon}\right\|_2$. Finally,  {\cite[Theorem~1.4]{Jeanjean2022Onglobal} implies} that such $u_{\varepsilon}$ has constant sign and, up to {translations, $u_{\varepsilon}$} is radial and radially {monotone.} \hfill$\Box$

\section{Proof of Theorem \ref{theorem1}}\label{Sec4}
	Throughout this section, we assume that
$(g_0)-(g_4)$ hold.	The {family
$(u_{\varepsilon})_\varepsilon$ and  the  related} Lagrange multipliers $(\lambda_{\varepsilon})_\varepsilon$ are given in Theorem~\ref{theoremPerturbed}, where $\varepsilon \in (0,1)$.

	\begin{Lemma}\label{Lemma41}

		The quantities $\left\|\nabla u_{\varepsilon}\right\|_2$, $\left\|\nabla u_{\varepsilon}\right\|_q$ and $\lambda_{\varepsilon}$ are bounded for $\varepsilon \in (0,1)$.
	\end{Lemma}	
	\begin{proof}
		By $(g_0), (g_1)$, and $(g_3)$, for every $\delta>0$ there exists $C_\delta>0$ such that for every $s \in \mathbb{R}$
		$$
		G_{+}(s) \leq C_\delta|s|^2+\delta|s|^{\bar{q}} .
		$$
Let us first consider the case $N \geq 3$. Since $\lambda_{\varepsilon}>0$, from \eqref{eqPoh} we obtain
		$$
		\frac{1}{2^*}\left\|\nabla u_{\varepsilon}\right\|_2^2+\frac{1}{q^*}\left\|\nabla u_{\varepsilon}\right\|_q^q<
		\int_{\mathbb{R}^N} G_{+}\left(u_{\varepsilon}\right) d x
		\leq C_\delta\left\|u_{\varepsilon}\right\|_2^2+\delta C_{N, \bar{q}}^{\bar{q}} c^{\bar{q} \left(1-\delta_{\bar{q}}\right)}\|\nabla u\|_2^{\bar{q} \delta_{\bar{q}}}.
		$$
		Taking $\delta<\left( C_{N, \bar{q}}^{\bar{q}} c^{\bar{q} \left(1-\delta_{\bar{q}}\right)}\right)^{-1}$, we obtain the boundedness of $\left\|\nabla u_{\varepsilon}\right\|_2$ and $\left\|\nabla u_{\varepsilon}\right\|_q$. In addition, this and \eqref{eqPoh} yield
		\begin{equation}\label{eq41}
			\begin{aligned}
				0&<\frac{1}{2} \lambda_{\varepsilon} c^2\\
				&<\int_{\mathbb{R}^N} \left(\frac{1}{2^*}\left|\nabla u_{\varepsilon}\right|^2+\frac{1}{q^*}\left|\nabla u_{\varepsilon}\right|^q+\frac{1}{2} \lambda_{\varepsilon}\left|u_{\varepsilon}\right|^2+ G_{-}^{\varepsilon}\left(u_{\varepsilon}\right) \right)d x\\
				&= \int_{\mathbb{R}^N} G_{+}\left(u_{\varepsilon}\right) d x \leq C<\infty,
			\end{aligned}
		\end{equation}
		and $(\lambda_{\varepsilon})_{\varepsilon}$ is bounded as well. Now, let us consider the case $N=2$. From \eqref{eqPoh} we get
		\begin{equation}\label{eq42}
			\begin{aligned}
				c^2 \lambda_{\varepsilon} &\leq \int_{\mathbb{R}^N} \left(\lambda_{\varepsilon}\left|u_{\varepsilon}\right|^2+2 G_{-}^{\varepsilon}\left(u_{\varepsilon}\right)\right) d x \\
				&\leq  \int_{\mathbb{R}^N} 2G_{+}\left(u_{\varepsilon}\right) d x\\
				&\leq C_\delta\left\|u_{\varepsilon}\right\|_2^2+\delta C_{N, \bar{q}}^{\bar{q}} c^{\bar{q} \left(1-\delta_{\bar{q}}\right)}\|\nabla u\|_2^{\bar{q} \delta_{\bar{q}}}.
			\end{aligned}
		\end{equation}
		Moreover, using \eqref{eqPoh} again, we obtain
		$$
		-\beta>J_{\varepsilon}\left(u_{\varepsilon}\right)=\frac{1}{2}\left(\left\|\nabla u_{\varepsilon}\right\|_2^2+\left\|\nabla u_{\varepsilon}\right\|_q^q-c^2 \lambda_{\varepsilon}\right),
		$$
		which, together with \eqref{eq42}, implies
		$$
		\frac{1}{2}\left\|\nabla u_{\varepsilon}\right\|_2^2+\frac{1}{2}\left\|\nabla u_{\varepsilon}\right\|_q^q \leq C_\delta c^2-\beta+\delta C_{N, \bar{q}}^{\bar{q}} c^{\bar{q} \left(1-\delta_{\bar{q}}\right)}\|\nabla u\|_2^{\bar{q}a \delta_{\bar{q}}}.
		$$
		Taking $\delta<\left( 2C_{N, \bar{q}}^{\bar{q}} c^{\bar{q} \left(1-\delta_{\bar{q}}\right)}\right)^{-1}$, we obtain the boundedness of 	$\left\|\nabla u_{\varepsilon}\right\|_2^2$ and $\left\|\nabla u_{\varepsilon}\right\|_q^q$. Hence, the boundedness of the family  $(\lambda_{\varepsilon})_{\varepsilon}$ follows from \eqref{eq42}.
	\end{proof}

	Note that for every $\varepsilon \in(0,1)$ and every $u \in \mathcal{D}(c)$ there holds $J_{\varepsilon}\left(u_{\varepsilon}\right) \leq J_{\varepsilon}(u) \leq J(u)$, whence $J_{\varepsilon}\left(u_{\varepsilon}\right) \leq \inf _{\mathcal{D}(c)} J=: m(c)$.
Now we {are ready to} complete the proof of Theorem \ref{theorem1}.\medskip

\noindent	\textbf{Proof of Theorem \ref{theorem1}.}
Let $\left(\varepsilon_n\right)_n$ be a sequence in $(0,1)$ such that $\varepsilon_n \rightarrow 0^{+}$ as $n \rightarrow \infty$, and let $\left(u_{\varepsilon_n}\right)_n$ be a sequence of solutions in Theorem \ref{theoremPerturbed}.
 From Lemma \ref{Lemma41}, there exist $u \in \mathcal{D}(c)$ and $\lambda \geq 0$ such that, up to a subsequence, $u_{\varepsilon_n} \rightharpoonup u$ in $X$ and $\lambda_{\varepsilon_n} \rightarrow \lambda$ as $n \rightarrow \infty$. Arguing as in \cite [Theorem 1.1]{Mederski2021optimal}, we obtain that $\varphi_{\varepsilon_n}\left(u_{\varepsilon_n}\right) g_{-}\left(u_{\varepsilon_n}\right) v \rightarrow g_{-}(u) v$ as $n\rightarrow \infty$ for a.e. $x\in \mathbb{R}^N$ for every $v \in \mathcal{C}_0^{\infty}\left(\mathbb{R}^N\right)$ and that $g_{+}\left(u_{\varepsilon_n}\right) v$ and $\varphi_{\varepsilon_n}\left(u_{\varepsilon_n}\right) g_{-}\left(u_{\varepsilon_n}\right) v$ are uniformly integrable (and tight). We deduce that for every $v \in \mathcal{C}_0^{\infty}\left(\mathbb{R}^N\right)$
$$
	-\lambda \int_{\mathbb{R}^N} u v d x \leftarrow J_{\varepsilon_n}^{\prime}\left(u_{\varepsilon_n}\right) v \rightarrow J^{\prime}(u) v,\quad {\mbox{as $n\rightarrow \infty$}},
	$$
	i.e.,
	$$
	-\Delta u-\Delta_q u+\lambda u=g(u), \quad x\in \mathbb{R}^N.
	$$
	Moreover, \eqref{eq41} (when $N \geq 3$) or \eqref{eq42} (when $N=2$ ) yields
	$$
	\int_{\mathbb{R}^N} G_{-}^{\varepsilon}\left(u_{\varepsilon}\right) d x \leq \int_{\mathbb{R}^N} G_{+}\left(u_{\varepsilon}\right) d x \leq C.
	$$
{Hence,} in view of Fatou's lemma, $G_{-}(u) \in L^1\left(\mathbb{R}^N\right)$. In particular, {as
shown in~\cite{Lcai2024Norm}, the couple} $(u, \lambda)$ satisfies the Pohozaev identity
	\begin{equation}\label{Lemma43}
		\frac{N-2}{2} \int_{\mathbb{R}^N}|\nabla u|^2 d x+\frac{N-q}{q} \int_{\mathbb{R}^N}|\nabla u|^q d x+\frac{\lambda N}{2} \int_{\mathbb{R}^N}|u|^2 d x=N \int_{\mathbb{R}^N} G(u) d x .
	\end{equation}
	We can assume that $\left\|\nabla u_{\varepsilon_n}\right\|_2$, $\left\|\nabla u_{\varepsilon_n}\right\|_q$ and $\int_{\mathbb{R}^N} G_{-}^{\varepsilon_n}\left(u_{\varepsilon_n}\right) d x$ are convergent as $n \rightarrow \infty$. Since
$$\int_{\mathbb{R}^N} G_{+}\left(u_{\varepsilon_n}\right) d x \rightarrow \int_{\mathbb{R}^N} G_{+}(u)dx,\,\,\text{as}\,\, n\rightarrow\infty,$$
 recall that each $u_{\varepsilon_n}$ is radially symmetric, we have $J(u) \leq \lim \limits_{n \rightarrow \infty} J_{\varepsilon_n}\left(u_{\varepsilon_n}\right)<0$ from Lemma \ref{Lemmam(c)<0}; in particular, $u \neq 0$.

If $\lambda=0$, from \eqref{Lemma43}, we have $J(u)=(\|\nabla u\|_2^2+\|\nabla u\|_q^q) / N>0$, therefore $\lambda>0$. We prove that $u_{\varepsilon} \rightarrow u$ as $\epsilon \rightarrow 0$ in $X$. Since $\lambda>0$, from \eqref{Lemma43} there follows
$$
	\begin{aligned}
	&\int_{\mathbb{R}^N}\left( \frac{N-2}{2}|\nabla u|^2+\frac{N-q}{q}|\nabla u|^2+\frac{N\lambda}{2} |u|^2+N G_{-}(u) \right) d x \\
		\leq& \lim _{\varepsilon \rightarrow0} \int_{\mathbb{R}^N} \left( \frac{N-2}{2}|\nabla u_{\varepsilon }|^2+\frac{N-q}{q}|\nabla u_{\varepsilon}|^2+\frac{N\lambda}{2} \left|u_{\varepsilon}\right|^2+N G_{-}^{\varepsilon}\left(u_{\varepsilon}\right) \right) d x \\
		= &\lim _{\varepsilon \rightarrow 0}N \int_{\mathbb{R}^N} G_{+}\left(u_{\varepsilon}\right) d x\\
		=& N \int_{\mathbb{R}^N} G_{+}(u) d x \\
		 =&\int_{\mathbb{R}^N}\left(  \frac{N-2}{2}|\nabla u|^2+\frac{N-q}{q}|\nabla u|^2+\frac{N\lambda}{2} |u|^2+N G_{-}(u) \right) dx.
	\end{aligned}
	$$
{This  yields that} $\left\|u_{\varepsilon}\right\|_2 \rightarrow\|u\|_2$; in particular, $u \in \mathcal{S}(c)$. Furthermore, in a similar way, we can obtain $J(u) \leq$ $\lim \limits _{\varepsilon\rightarrow0} J_{\varepsilon}\left(u_{\varepsilon}\right) \leq m(c) \leq J(u)$, which shows that $\left\|\nabla u_{\varepsilon}\right\|_2 \rightarrow\|\nabla u\|_2$ and $J(u)=m(c)$.

Since \eqref{Equationperturbed}  is translation-invariant, we can assume that every $u_{\varepsilon}$ is radial, hence so is $u$. Moreover, because of $u \neq 0$, there exists $x \in \mathbb{R}^N$ such that $u(x) \neq 0$ and so $u_{\varepsilon}(x)$ has the same sign as $u(x)$ for every $\varepsilon$ sufficiently small. {Since} every $u_{\varepsilon}$ has constant sign, it has everywhere the same sign as $u(x)$, hence $u$ has constant sign too. Finally, if there exist $x, y, z \in \mathbb{R}^N$ such that $|x|<|y|<|z|$ and either $u(y)<\min \{u(x), u(z)\}$ or $u(y)>\max \{u(x), u(z)\}$, then arguing as before, we obtain a contradiction. \hfill$\Box$

\begin{remark}\label{re111}
	(i) The proof of Theorem \ref{theorem1} contains the relevant result that $m_{\varepsilon}(c) \rightarrow m(c)$ as $\varepsilon \rightarrow 0^{+}$.\\ (ii) Unlike the proof of Theorem \ref{theorem2}, we cannot use the information $\lambda>0$ to deduce $u \in \mathcal{S}(c)$, because we do not know whether $u$ is a critical point of $\left.J\right|_{\mathcal{D}(c)}$.
\end{remark}	

\noindent	\textbf{Proof of Proposition \ref{prop1}.}
(i) We fix $c>0$ and take some function $u \in\mathbb{S}(c) \cap\mathcal{C}_0^{\infty}\left(\mathbb{R}^N\right) \backslash\{0\}$. For $s>0$, let $u_t(x)=t^{N / 2} u(t x)$. Then, we see that $u_t \in \mathcal{S}(c)$. By the assumption in (i), there exist a positive constant $\delta$ such that
$$
G(s) \geq C|s|^{\tilde{q}}, \quad \text { if }|s|<\delta,
$$
where $C$ is a constant determined by
$$
C=\int_{\mathbb{R}^N}\left(|\nabla u|+|\nabla u|^q\right) d x / \int_{\mathbb{R}^N}|u|^{\tilde{q}} d x.
$$
Hence $G\left(\left|u_t\right|\right) \geq C\left|u_t\right|^{\tilde{q}}$ holds for the sufficiently small $s>0$. Thus, we have
\begin{equation}
	\begin{aligned}
		J\left(u_t\right) &=\frac{1}{2} \int_{\mathbb{R}^N}\left|\nabla u_t\right|^2 +\frac{1}{q} \int_{\mathbb{R}^N}\left|\nabla u_t\right|^qdx-
		\int_{\mathbb{R}^N}G(u_t)dx\\
		&\leq\frac{1}{2} \int_{\mathbb{R}^N}\left|\nabla u_t\right|^2 +\frac{1}{q} \int_{\mathbb{R}^N}\left|\nabla u_t\right|^qdx-
		C\int_{\mathbb{R}^N}\left|u_t\right|^{\tilde{q}}dx\\
		&\leq\frac{1}{2}t^2 \int_{\mathbb{R}^N}\left|\nabla u\right|^2 +\frac{1}{q} t^q\int_{\mathbb{R}^N}\left|\nabla u\right|^qdx-
		t^{\tilde{q}\delta_{\tilde{q}}}C\int_{\mathbb{R}^N}
			\left|u\right|^{\tilde{q}}dx\\
		&\leq\left(\frac{1}{2}t^2-	t^{\tilde{q}\delta_{\tilde{q}}}\right) \int_{\mathbb{R}^N}\left|\nabla u\right|^2 +
		\left(\frac{1}{q}t^{q(1+\delta_q)}-	t^{\tilde{q}\delta_{\tilde{q}}}\right) \int_{\mathbb{R}^N}\left|\nabla u\right|^qdx\\
	\end{aligned}
\end{equation}
Since $\tilde{q}\delta_{\tilde{q}}=\min\{2,q(1+\delta_q)\}$,
it concludes that $m(c) \leq J\left(u_t\right)<0$ for any $c>0$.\\
(ii) By the assumption in (ii), there exists a positive constant $C_g$ depending on $g$ such that
$$
G(t) \leq C_g|t|^{\tilde{q}}
$$
for any $g\geq 0$. For $u \in \mathcal{S}(c)$, using the Gagliardo-Nirenberg inequality, we have
$$
\int_{\mathbb{R}^N} G(u) d x \leq C_g\|u\|_{\tilde{q}}^{\tilde{q}} \leq C_g C_{N,q}^{\tilde{q}}c^{2 / N}\|\nabla u\|_{2}^2.
$$
For a sufficiently small $c>0$, it can be shown that $C_gC_{N,q}c^{2/N} \leq 1/2$ holds. After choosing an appropriately small $c$, we have
$$
J(u) \geq \frac{1}{2}\|\nabla u\|_{2}^2+\frac{1}{q}\|\nabla u\|_{q}^q-\frac{1}{2}\|\nabla u\|_{2}^2\geq 0.
$$
This means $m(c) \geq 0$ for a small $\alpha>0$. Hence, we obtain $\bar{c}>0$.

\noindent	\textbf{Proof of Proposition \ref{prop2}.} It follows from Lemma \ref{Lemmam(c)<0} and Remark \ref{re111} (i).

\noindent	\textbf{Proof of Theorem \ref{theorem2}.}
First of all, note that $G$ satisfies the assumptions $({g}_0)-({g}_4)$,  and that such a solution exists if and only if $ (g_2) $ holds: the `if' part follows from Theorem \ref{theorem1}, while the `only if' part follows from the fact that, since $G(u) \in L^1\left(\mathbb{R}^N\right)$ and
$$
\int_{\mathbb{R}^N}\frac{(N-2)}{2}|\nabla u|^2+\frac{(N-q)}{q}|\nabla u|^q+\frac{N}{2} \lambda|u|^2 d x= N \int_{\mathbb{R}^N} G(u) d x.
$$
Let $s>0$. Clearly, $G(s)>0$ if and only if
$$
\widetilde{G}(s):=\frac{\alpha}{2}\left(\ln s^2-1\right)+\frac{\mu}{p} s^{p-2}>0 .
$$
Since
$$
\widetilde{G}^{\prime}(s)=\frac{\alpha}{s}+\mu \frac{p-2}{p} s^{p-3},
$$
we have,
$$
\max_{s>0} \widetilde{G}(s)=\frac{\alpha}{2}\left(\ln \left(\frac{\alpha p}{\mu(2-p)}\right)^{2 /(p-2)}-1\right)-\frac{\alpha}{p-2}.
$$
Observing that $ (g_2) $ holds if and only if $\max_{s>0} \widetilde{G}(s)>0$. Hence, the proof is concluded.
	
\section{Proof of Theorem \ref{thm2muty}}\label{Sec5}
	In this section, we first present some fundamental concepts and properties concerning {the} Orlicz
	spaces. For {more details we}, refer to \cite{Adams2003Sobolev,Fukagai2006NormOrlicz,Rao1985Orlicz}

	\begin{definition}
{\rm		An {\em $N$-function $A$} is a nonnegative continuous function $A:\mathbb{R} \rightarrow[0,\infty)$ that satisfies the following conditions:\\
		$\phantom{i}(i)$ $A$ is convex and even.\\
		$(ii)$ $\lim \limits _{t \rightarrow 0} \frac{A(t)}{t}=0$ and $\lim \limits _{t \rightarrow \infty} \frac{A(t)}{t}=\infty$.}
	\end{definition}
	An N-function $A$ is said to satisfy {the} $\Delta_2$ {\em condition globally} if and only if there exists $K>0$ such that
	$$A(2 s) \leq K A(s), \quad \text{for all}\ s \in \mathbb{R}.$$
	An N-function $A$ is said to satisfy the
$\nabla_2$ {\em condition globally} if and only if there exists {$\ell>1$} such that
	$$2 \ell A(s) \leq A(\ell s), \quad \text{for all}\ s \in \mathbb{R}.$$
	Equivalently, $A$ is {differentiable, with $A'=a$,} satisfies the $\Delta_2$ (respectively, $\nabla_2$ ) condition globally if and only if there exists $C>1$ such that
	\begin{equation}\label{Eqdetla}
		C A(s) \geq s a(s) \text { (respectively, } s a(s) \geq C A(s)\text { ),  for all } s \in \mathbb{R} .
	\end{equation}
From now on, $A$ {denotes} the same function {given  in assumtion $(A)$.}

{The Orlicz space $V$ associated with $A$ in $(A)$ is given by}
	$$
	V:=\left\{u \in L_{\text {loc }}^1\left(\mathbb{R}^N\right): A(u) \in L^1\left(\mathbb{R}^N\right)\right\},
	$$
 endowed with the norm
	$$
	\|u\|_V:=\inf \left\{\kappa>0: \int_{\mathbb{R}^N} A(u / \kappa) d x \leq 1\right\}.
	$$
	Since $A$ is an $N$-function that satisfies the $\Delta_2$ and $\nabla_2$ conditions globally, $\left(V,\|\cdot\|_V\right)$ is a reflexive Banach space. Moreover, {we will} use the following lemma.
	
	\begin{Lemma}\label{LemmaA(u)conver}
		(i) Let $u_n, u \in V$. Then $\lim \limits _{n \rightarrow \infty}\left\|u_n-u\right\|_V=0$ if and only if  $\lim \limits _{n \rightarrow \infty} \int_{\mathbb{R}^N} A\left(u_n-u\right) d x=0$.\\
		(ii) Let $\bar{V} \subset V$. Then $\bar{V}$ is a bounded {set} {in $V$} if and only if
		$$
		\left\{\int_{\mathbb{R}^N} A(u) d x: u \in \bar{V}\right\}
		$$
		is bounded {in $\mathbb R$}.\\
		(iii) Let $u_n, u \in V$. If $u_n \rightarrow u$ for a.e. $x\in \mathbb{R}^N$and $\int_{\mathbb{R}^N} A\left(u_n\right) d x \rightarrow \int_{\mathbb{R}^N} A(u) d x$, then $\left\|u_n-u\right\|_V \rightarrow 0.$
	\end{Lemma}
	\begin{proof}
		(i) and (ii) follow from \cite{Rao1985Orlicz}, while (iii) is a consequence of (i), (ii) and \cite{Brezis1983BL}.
	\end{proof}

{The space $W$ introduced in \eqref{W} is exactly
given by} $W=X \cap V$. Since {$X=(X,\|\cdot\|_X)$
and}  $\left(V,\|\cdot\|_V\right)$ {are reflexive Banach spaces}, so is $\left(W,\|\cdot\|_W\right)$, with the norm
	$$
	\|u\|_W:={\|u\|_X+\|u\|_V=}\|\nabla u\|_2+\|u\|_2+\|\nabla u\|_q+\|u\|_V.
	$$
{The next result is a} variant of Lions lemma.
	\begin{Lemma}\label{LemmaA(u)lions}
		Suppose that {$(u_n)_n$  is a bounded sequence
in~$W$  and that} for some $r>0$
		$$
		\lim _{n \rightarrow \infty} \sup _{y \in \mathbb{R}^N} \int_{B_r(y)}\left|u_n\right|^2 d x=0 .
		$$
Then $u_n \rightarrow 0$ in $L^p\left(\mathbb{R}^N\right)$ for every $p \in\left[2,q^{\prime}\right)$, where {$q^{\prime}
=\max \left\{2^*, q^*\right\}$}.
	\end{Lemma}	
	\begin{proof}
{Lemma~\ref{LemmaLions} implies that} $u_n \rightarrow 0$ in $L^m\left(\mathbb{R}^N\right)$ for every $m \in\left(2,q^{\prime}\right)$. We {claim that the strong convergence occurs also} in $L^2\left(\mathbb{R}^N\right)$. Take any $m \in\left(2,q^{\prime}\right)$ and $\varepsilon>0$. {Then, there exists} $\delta>0$ such that
		$$
		\begin{aligned}
			& |s|^2 \leq \varepsilon A(s), \quad \text { if }\,|s| \in[0, \delta], \\
			& |s|^2 \leq \delta^{2-m}|s|^m, \quad \text { if }\,|s|>\delta.
		\end{aligned}
		$$
Hence, we get
\begin{align*}
		\limsup _{n \rightarrow \infty} \int_{\mathbb{R}^N}\left|u_n\right|^2 d x \leq \varepsilon \limsup _{n \rightarrow \infty} \int_{\mathbb{R}^N} A\left(u_n\right) d x+\delta^{2-m} \limsup _{n \rightarrow \infty} \int_{\mathbb{R}^N}\left|u_n\right|^m d x=\varepsilon \limsup _{n \rightarrow \infty} \int_{\mathbb{R}^N} A\left(u_n\right) d x .
\end{align*}
{Letting $\varepsilon \rightarrow 0^+$, we  conclude the
proof thanks to} Lemma \ref{LemmaA(u)conver} (ii).
	\end{proof}
	
	Let $\mathcal{O}$ be any subgroup $\mathcal{O}(N)$ such that $\mathbb{R}^N$ is compatible with $\mathcal{O}$ (see \cite{Lions1982sym}), i.e., $\lim \limits _{|y| \rightarrow \infty} m(y, r)=$ $\infty$ for some $r>0$, where, for
{all} $y \in \mathbb{R}^N$,
	$$
	m(y, r):=\sup \left\{n \in \mathbb{N} \text { : there exist } g_1, \ldots, g_n \in \mathcal{O} \text { such that } {B_r}\left(g_i y\right) \cap B_r\left(g_j y\right)=\emptyset \text { for } i \neq j\right\} \text {. }
	$$
	In view of \cite{Lions1982sym}, {the space} $H_{\mathcal{O}}^1\left(\mathbb{R}^N\right)$ embeds compactly into $L^m\left(\mathbb{R}^N\right)$ for
{all $m$, with} $2<m<2^{*}$. Using the same argument, we
 obtain that $W_{\mathcal{O}}$, {defined in~\eqref{WO},}
embeds compactly into $L^m\left(\mathbb{R}^N\right)$ for
{all $m$, with} $2\leq m<q^{\prime}$, {as proved in the next result.}

	\begin{corollary}\label{CorLoc}
		Suppose that {$(u_n)_n$  is a bounded sequence
in~$W_{\mathcal{O}}$  and that}
 $u_n \rightarrow 0$ in $L_{\text {loc }}^2\left(\mathbb{R}^N\right)$. Then $u_n \rightarrow 0$ in $L^m\left(\mathbb{R}^N\right)$ for every $m \in\left[2,q^{\prime}\right)$.
	\end{corollary}

\begin{proof}
		Suppose that
		\begin{equation}\label{EqB(0)}
			\int_{B_1\left(y_n\right)}\left|u_n\right|^2 d x \geq C>0,
		\end{equation}
		for some sequence {$(y_n)_n \subset \mathbb{R}^N$ and a suitable} constant $C$. Observe that in the family $\left\{B_1\left(h y_n\right)\right\}_{h \in \mathcal{O}}$, we find an increasing number of disjoint balls provided that $\left|y_n\right| \rightarrow \infty$. Since {$(u_n)_n$} is bounded in $L^2\left(\mathbb{R}^N\right)$ and invariant with respect to $\mathcal{O}$, by \eqref{EqB(0)}
{the sequence $(y_n)_n$} must be bounded. Then, for sufficiently large $r$ {we obtain}
		$$
		\int_{B_r(0)}\left|u_n\right|^2 d x \geq \int_{B_r\left(y_n\right)}\left|u_n\right|^2 d x \geq c>0.
		$$
{This contradicts} \eqref{EqB(0)}. Therefore,
{the conclusion holds thanks to}
Lemma~\ref{LemmaA(u)lions}.
	\end{proof}

Then, we have the following result.
	
	\begin{proposition}
		If $(A)$ and $(f_0)-(f_3)$ hold, then the functional $\left.J\right|_W: W \rightarrow \mathbb{R}$ is of class $\mathcal{C}^1$. {Moreover, $\left.J\right|_{W \cap \mathcal{S}(c)}^{\prime}(u)=0$, with $u \in W \cap \mathcal{S}(c)$,}
if and only if there exists $\lambda \in \mathbb{R}$ such that $(u, \lambda)$ is a solution of~\eqref{Equation}.
	\end{proposition}	
	
	\begin{proof}
		Let $B: \mathbb{R} \rightarrow \mathbb{R}$ be the complementary {of the}  $N$-function of $A$. From \cite[Theorem II.V.8]{Rao1985Orlicz}, {the
function $B$} satisfies the $\Delta_2$ and $\nabla_2$ conditions globally as $A$ also satisfies these conditions. Let us recall that $V^{*}$, the dual space of $V$, is isomorphic to the Orlicz space associated with $B$(see \cite[Definition III.IV.2, Corollary III.IV.5, and Corollary IV.II.9]{Rao1985Orlicz}). From \cite[Theorem I.III.3]{Rao1985Orlicz}, we have
		$$
		B(a(s))=s a(s)-A(s) \leq(C-1) A(s),
		$$
		where $C>1$ is the constant given in the characterization of the $\Delta_2$ condition \eqref{Eqdetla}. Then, for every $u, v \in V$, we have
		$$
		\left|\int_{\mathbb{R}^N} a(u) v d x\right| \leq \int_{\mathbb{R}^N}|a(u)||v| d x \lesssim\|a(u)\|_{V^{*}}\|v\|_V .
		$$
Now we {proceeding as in the proof of Lemma~2.1 in \cite{Clement2000mpt}, we} obtain that,  under
the assumptions $(f_1)-(f_3)$,
both functionals $u \mapsto\int_{\mathbb{R}^N} A(u) d x$
and $u \mapsto \int_{\mathbb{R}^N} F(u) d x$  are of
class $\mathcal{C}^1(V) \subset \mathcal{C}^1(W)$. {The remaining proof is obvious, and we omit it.}
	\end{proof}	

To prove Theorem \ref{thm2muty}, we need some important
{preliminary results. As usually, we adopt the
notations} $F_{-}:=F_{+}-F$ and $f_{-}:=F_{-}^{\prime}$.

	\begin{Lemma}\label{EqA(u)bounded}
		If $(A), (f_0)-(f_3)$, and \eqref{etaC} hold, then $\left.J\right|_{W \cap \mathcal{S}(c)}$ is bounded below.
	\end{Lemma}

	\begin{proof}
From \eqref{eqeta}, for every $\delta>0$, there exists $C_\delta>0$ such that for every $s \in \mathbb{R}$
		$$
		F_{+}(s) \leq C_\delta|s|^2+(\eta+\delta)|s|^{\bar{q}} .
		$$
{Hence,} for every $u \in W \cap \mathcal{S}(c)$ there holds
$$
		\begin{aligned}
			J(u) & \geq \int_{\mathbb{R}^N}\left( \frac{1}{2}|\nabla u|^2+\frac{1}{q}|\nabla u|^q+A(u)-F_{+}(u) \right) d x \\
			& \geq \int_{\mathbb{R}^N} \left(\frac{1}{2}|\nabla u|^2+\frac{1}{q}|\nabla u|^q+A(u)-C_\delta|u|^2-(\eta+\delta)|u|^{\bar{q}} \right)  d x \\
			& \geq \frac{1}{2}\|\nabla u\|_2^2+\frac{1}{q}\|\nabla u\|_q^q+\int_{\mathbb{R}^N} A(u) d x-(\eta+\delta) C_{N, \bar{q}}^{\bar{q}} c^{\bar{q} \left(1-\delta_{\bar{q}}\right)}\|\nabla u\|_2^{\bar{q} \delta_{\bar{q}}} .
		\end{aligned}
		$$
Taking $\delta$ {so small  that} $2(\eta+\delta) C_{N, \bar{q}}^{\bar{q}} c^{\bar{q} \left(1-\delta_{\bar{q}}\right)} <1$,  we conclude that $J(u)$ bounded from below on $W \cap \mathcal{S}(c)$, {since $A(s)$ is nonnegative.}
	\end{proof}

	\begin{remark}\label{ref+}
{\rm		Setting $f_{\mathrm{p}}:=\max \{f, 0\}$ and $f_{\mathrm{n}}:=\max \{-f, 0\}$, {we see that}
		\begin{equation*}
			F_{+}(s)= \begin{cases}\int_0^s \max \{f(t), 0\} d t \quad \text { if } s \geq 0, \\ \int_s^0 \max \{-f(t), 0\} d t \quad \text { if } s<0 .\end{cases}
		\end{equation*}
{Hence, being}  $f_{+}(s)=F_{+}^{\prime}(s)$, $f_{-}(s):=f_{+}(s)-f(s)$, and $F_{-}(s):=F_{+}(s)-F(s) \geq 0$ for $s \in \mathbb{R}$,
{it holds} $f_{-}(s)=\max \{-f(s), 0\}=f_{\mathrm{n}}(s)$ if $s \geq 0$ and $f_{-}(s)=-\max \{f(s), 0\}=-f_{\mathrm{p}}(s)$ if $s<0$. {Therefore,} we conclude that, $f_{-}(s) s \geq 0$ for every $s \in \mathbb{R}$.}
	\end{remark}
	
	\begin{Lemma}\label{Lema(s)}
		If $(A)$ {holds}, then {$s\mapsto a(s)s$} is nonnegative for all $s \in \mathbb{R}$.
	\end{Lemma}
	
	\begin{proof}
		Since {$A$} is an even function, {$s\mapsto a(s)s$ also is an} even function. {Thus,}
 we only need to {prove} that {$s\mapsto a(s)s$} is nonnegative on $\left(0,\infty\right)$.
		
		Suppose, by contradiction, {that
there exists} $s_1>0$, such that $a(s_1)s_1<0$,
then there {exists}
		$s_2 \in \left(0,s_1\right)$ such that $a(s_2)s_2>0$. {Otherwise} $a(s)\leq 0$ on $(0, s_1)$, and thus $A(s_1)<0$, {which
is a contradiction}. Since {$s\mapsto a(s)s$} is convex
{in $\mathbb R$ by$(A)$, the following inequality
holds}
		$$
		\frac{a(0) 0-a\left(s_1\right) s_1}{0-s_1} \geq \frac{a(0) 0-a\left(s_2\right) s_2}{0-s_2}.
		$$
{Hence, $a(s_1)\geq a(s_2)$, which is the required} contradiction.
	\end{proof}
	
	\begin{Lemma}\label{LemA(u)PS}
		If $(A)$ and $(f_0)-(f_3)$ hold, then $\left.J\right|_{{W_{\mathcal{O}}} \cap \mathcal{S}(c)}$ satisfies the Palais-Smale condition.
	\end{Lemma}
	\begin{proof}
		Let {$(u_n)_n$ be a sequence in}
  $W_{{\mathcal{O}}} \cap \mathcal{S}(c)$ such that {$(J\left(u_n\right))_n$ is bounded in $\mathbb R$} and $\left.J\right|_{W_{\mathcal{O}} \cap \mathcal{S}(c)} ^{\prime}\left(u_n\right) \rightarrow 0$
  {as $n\to\infty$}. From
  Lemma~\ref{EqA(u)bounded} and Lemma~\ref{LemmaA(u)conver}
  (ii), {the sequence $(u_n)_n$} is bounded in $W$. {Therefore,}
   there exists $u \in W_{\mathcal{O}}$ such that $u_n
  \rightharpoonup u$ in $W_{\mathcal{O}}$, up to a subsequence. Then,  $u_n \rightarrow u$ in
  $L^m\left(\mathbb{R}^N\right)$ for every $m \in\left[2,q^{\prime}\right)$ {by
   Corollary~\ref{CorLoc}.
In} particular, $u \in \mathcal{S}(c)$. Up to a  subsequence again, we can assume that $u_n \rightarrow u$ for a.e.
{in} $ \mathbb{R}^N$. Additionally, from \cite[{Lemma~3}]{Berestycki1983infinite}, there
{exists a sequence $(\lambda_n)_n$ in} $\mathbb{R}$ such that
		\begin{equation}\label{EqA(u)lambda}
			-\Delta u_n+\lambda_n u_n-g\left(u_n\right)  \rightarrow 0 \quad \text { in } W_{\mathcal{O}}^{*},
		\end{equation}
		where $W_{\mathcal{O}}^{*}$ is the dual space of $W_{\mathcal{O}}$. Testing \eqref{EqA(u)lambda} with {$(u_n)_n$}, we obtain that {$(\lambda_n)_n$} is bounded as well. {Hence,} there exist $\lambda \in \mathbb{R}$ such that $\lambda_n \rightarrow \lambda$ up to a subsequence, and $(u, \lambda)$ is a solution {of}~\eqref{Equation}. Finally, from the Nehari identity and the fact that $\lambda_n \rightarrow \lambda$, we obtain
{that}
		$$
		\begin{aligned}
			\int_{\mathbb{R}^N} \left( |\nabla u|^2+|\nabla u|^q+a(u) u+f_{-}(u) u \right)d x & =\int_{\mathbb{R}^N}\left(f_{+}(u) u-\lambda|u|^2\right)  d x\\
			&=\lim _{n \rightarrow\infty} \int_{\mathbb{R}^N}\left( f_{+}\left(u_n\right) u_n-\lambda\left|u_n\right|^2 \right)d x \\
			& =\lim _{n \rightarrow \infty} \int_{\mathbb{R}^N}
			\left(\left|\nabla u_n\right|^2+|\nabla u|^q+a\left(u_n\right) u_n+f_{-}\left(u_n\right) u_n \right) d x,
		\end{aligned}
		$$
{Now,} $f_{-}(s) s \geq 0$ for every
		$s \in \mathbb{R}$ 		by Remark~\ref{ref+}, {
so that} by Fatou's lemma
		$$\limsup _{n \rightarrow \infty} \int_{\mathbb{R}^N} f\left(u_n\right) u_n d x \geq \int_{\mathbb{R}^N} f(u) u d x.$$
{Similarly, $a(s)s\geq0$ for every
		$s \in \mathbb{R}$ 	by $(A)$ and by Lemma \ref{Lema(s)}, so that}
$$\limsup _{n \rightarrow \infty} \int_{\mathbb{R}^N} a\left(u_n\right) u_n d x \geq \int_{\mathbb{R}^N} a(u) u d x.$$
{Additionally,} since $u_n \rightharpoonup u$ in $W_{\mathcal{O}}$,
		$$\liminf _{n \rightarrow \infty} \left\|\nabla u_n\right\|_2^2 \geq \|\nabla u\|_2^2
		\quad
		\text{and} \quad
		\liminf _{n \rightarrow \infty} \left\|\nabla u_n\right\|_q^q \geq \|\nabla u\|_q^q.$$
		The above  {properties imply} that, up to subsequences, as $n \rightarrow \infty$
\begin{gather*}\left\|\nabla u_n\right\|_2^2 \rightarrow\|\nabla u\|_2^2,\quad \left\|\nabla u_n\right\|_q^q \rightarrow\|\nabla u\|_q^q,\\
\int_{\mathbb{R}^N} f\left(u_n\right) u_n d x \rightarrow
 \int_{\mathbb{R}^N} f(u) u d x\quad\mbox{and}\quad
\int_{\mathbb{R}^N} a\left(u_n\right) u_n d x \rightarrow \int_{\mathbb{R}^N} a(u) u d x.\end{gather*}
It remains to prove that $\left\|u_n-u\right\|_V \rightarrow 0$. In virtue of Lemma \ref{LemmaA(u)conver} (iii), it suffices to {show} that $\int_{\mathbb{R}^N} A\left(u_n\right) d x \rightarrow \int_{\mathbb{R}^N} A(u) d x$. Additionally, since $A$ satisfies the $\nabla_2$ condition globally, we have that
		$s a(s) \geq C A(s) $.
		Since $\int_{\mathbb{R}^N} a\left(u_n\right) u_n d x \rightarrow \int_{\mathbb{R}^N} a(u) u d x$, by a variant of the Lebesgue Dominated Convergence theorem, we have that  $\int_{\mathbb{R}^N} A\left(u_n\right) d x \rightarrow \int_{\mathbb{R}^N} A(u) d x$.
	\end{proof}
	\begin{Lemma}\label{LemGamma}
		For every $k \geq 1$ there exist
{the functions}
$\gamma_k: \mathbb{S}^{k-1} \rightarrow W_{\mathcal{O}(N)}
\cap \mathcal{S}(c)$ and $\tilde{\gamma}_k: \mathbb{S}^{k-1}
\rightarrow  \mathcal{X} \cap \mathcal{S}(c)$
{which are} odd and continuous.
	\end{Lemma}

	\begin{proof}
		Fix $k \in \mathbb{N}$. In view of \cite[Lemma 3.4]{Jeanjean2019Nonradial}, there exist constants $R(k)>2(k+1)$ and $c_k>0$ such that, for any $R \geqslant R(k)$,
		there exists an odd continuous mapping $\tau_{k, R}: \mathbb{S}^{k-1} \rightarrow H^1\left(\mathbb{R}^N\right)$ having the properties that $\tau_{k, R}[\sigma]$ is a radial function, $\operatorname{supp}\left(\tau_{k, R}[\sigma]\right) \subset \bar{B}_R(0)$ for any $\sigma \in \mathbb{S}^{k-1}$. Therefore, $\tau_{k, R}$ is also
		an odd continuous mapping from $\mathbb{S}^{k-1}$ to $W_{\mathcal{O}(N)}$.
		
	Fix $c>0$ and $k \in \mathbb{N}$, we can define an odd continuous mapping $\gamma_k: \mathbb{S}^{k-1} \rightarrow W_{\mathcal{O}(N)} \cap \mathcal{S}(c)$ as follows:
		$$
		{\gamma_{k}}[\sigma](x):=\tau_k[\sigma]\left(c^{-1 / N} \cdot\left\|\tau_k[\sigma]\right\|_{L^2\left(\mathbb{R}^N\right)}^{2 / N} \cdot x\right), \quad x \in \mathbb{R}^N \text { and } \sigma \in \mathbb{S}^{k-1} .
		$$
		Let $\chi: \mathbb{R} \rightarrow[0,1]$ be an odd smooth function such that $\chi(t)=1$ for any $t \geqslant 1$.
{Let us introduce}
		$$
		\pi_{k, R}[\sigma](x):=\tau_{k, R}[\sigma](x) \cdot \chi\left(\left|x_1\right|-\left|x_2\right|\right),
		$$
		where $\sigma \in \mathbb{S}^{k-1}$ and $x=\left(x_1, x_2, x_3\right) \in \mathbb{R}^M \times \mathbb{R}^M \times \mathbb{R}^{N-2 M}$. Clearly, $\pi_{k, R}$ is an odd continuous mapping from $\mathbb{S}^{k-1}$ to $\mathcal{X}$.
		
		Fix $c>0$ and $k \in \mathbb{N}$, we can define an odd continuous mapping $\gamma_{ k}: \mathbb{S}^{k-1} \rightarrow \mathcal{X} \cap \mathcal{S}(c) $ as follows:
		$$
		\bar{\gamma_{k}}[\sigma](x):=\pi_k[\sigma]\left(c^{-1 / N} \cdot\left\|\pi_k[\sigma]\right\|_{L^2\left(\mathbb{R}^N\right)}^{2 / N} \cdot x\right), \quad x \in \mathbb{R}^N \text { and } \sigma \in \mathbb{S}^{k-1},
		$$
{as required.}
	\end{proof}

{In what follows,	
we} make use of the {next} abstract theorem, where $\mathcal{G}$ stands for the Krasnoselsky genus \cite[Chapter 5]{Struwe2008Variational}.

	\begin{theorem}[{\cite[Theorem 2.1]{Jeanjean2019Nonradial}}]\label{Thmmupt}
		Let $E$ be a Banach space, {let $H \subset E$ be}
a Hilbert space with scalar product $(\cdot \mid \cdot)$.
{Fix $R>0$ and put} $\mathcal{M}:=\{u \in E\,:\,(u \mid u)=R\}$. {Asumme that}  $I \in \mathcal{C}^1(E)$
{is an even functional and that } $\left.I\right|_{\mathcal{M}}$ is bounded below. For every $k \geq 1$, {set}
		$$
		c_k:=\inf _{A \in \Gamma_k} \sup _{u \in A} I(u), \quad \Gamma_k:=\{A \subset \mathcal{M}: A=-A=\bar{A} \text { and } \mathcal{G}(A) \geq k\} .
		$$
		Then, for every $k \geq 1$ and $-\infty<c_1 \leq \cdots \leq c_k \leq c_{k+1} \leq \ldots$ the following holds: if there exist $k \geq 1$ and $h \geq 0$ such that $c_k=\cdots=c_{k+h}<\infty$ and $\left.I\right|_{\mathcal{M}}$ satisfies the Palais-Smale condition at the level $c$, then
		$$
		\mathcal{G}\left(\left\{u \in \mathcal{M}: I(u)=c_k \text { and }\left.I\right|_{\mathcal{M}} ^{\prime}(u)=0\right\}\right) \geq h+1
		$$
		(in particular, taking $h=0$, every $c_k$ is a critical value of $\left.I\right|_{\mathcal{M}}$ ).
	\end{theorem}

Hence we can conclude in the following way.
	
\begin{proof} [Proof of Theorem \ref{thm2muty}]
{Take}  $E=W_{\mathcal{O}(N)}$ (respectively,
$E=\mathcal{X}$), $H=L^2\left(\mathbb{R}^N\right)$,
$R=c^2, \mathcal{M}=$ $\mathcal{S}(c) \cap E$, and
$I=\left.J\right|_E$. From Lemmas \ref{EqA(u)bounded}
and \ref{LemA(u)PS} {the functional}
$\left.I\right|_{\mathcal{M}}=J_{\mathcal{S}(c) \cap E}$ is
bounded below and satisfies the Palais-Smale {condition.}
{Moreover,  Lemma \ref{LemGamma} implies that}
$\gamma_k\left(\mathbb{S}^{k-1}\right) \in \Gamma_k$
 (respectively, $\tilde{\gamma}_k\left(\mathbb{S}^{k-1}\right)
 \in \Gamma_k$) for every $k$. {Thus,} the numbers $c_k$
are finite. Applying Theorem \ref{Thmmupt}, we conclude
{there exist  infinitely many solutions of~\eqref{Equation}.}

	Concerning the existence of a least-energy solution, we
 consider a {minimizing sequence $(u_n)_n$ in
  $\mathcal{S}(c) \cap E$} such that $\lim \limits_{n
  \rightarrow \infty} {J}\left(u_n\right)=\inf
  _{\mathcal{S}(c) \cap E} J$. From Ekeland's variational
 principle, we can assume that
{$(u_n)_n$} is a Palais-Smale sequence for
$\left.J\right|_{\mathcal{S}(c) \cap E}$.
{Hence,  arguing as above, we} obtain a solution
$(\bar{u}, \bar{\lambda}) \in \mathbb{R} \times(\mathcal{S}(c) \cap E)$ of~\eqref{Equation} such that $J(\bar{u})=\min _{\mathcal{S}(c) \cap E} J$. The fact that $\min _{\mathcal{S}(c) \cap W_{\mathcal{O}(N)}} J=\min _{\mathcal{S}(c) \cap W} J$ follows from the properties of the Schwartz rearrangement.
\end{proof}

	\section{Appendix}
{In this section we show that the energy functional $J$ is not well defined in $X$.}

\begin{remark}\label{Re-2N/N+2}
{\rm	Since we consider the case where
	\begin{equation}
		2 < \bar{q} = \left(1 + \frac{2}{N}\right) \min \{2, q\},
	\end{equation}
	we have either $\frac{2N}{N+2} < q < 2, \quad N \geq 2$ or $2 < q < N, \quad N \geq 3$.}
\end{remark}

	\begin{Lemma}\label{lemmanotwelldefined}
		For $ \frac{2 N}{N+2}<q<2$, $N \geq 2 $ or $2<q<N$, $N \geq 3$,  there exists $u \in X$ such that {$\int_{\mathbb{R}^N} u^2 \log u^2 d x=-\infty$}.
\end{Lemma}

\begin{proof}
{It} is enough to consider a smooth function that satisfies
	$$
	u(x)=\begin{cases}
		\left(|x|^{N / 2} \log (|x|)\right)^{-1}, &|x| \geq 3, \\
		0, &|x| \leq 2.
	\end{cases}
	$$
	We know that $u\in H^1(\mathbb{R}^N)$ and $\int_{\mathbb{R}^N} u^2 \log u^2 d x=-\infty$ for $N\geq 2$,
	so we only need to prove {that} $u \in D^{1, q}(\mathbb{R}^N)$, that is $\int_{\mathbb{R}^N}|\nabla u|^q d x<\infty$.
	
	If $2<q<N$, $N \geq 3$, choose $R>0$ large enough such
that $u(x)<1$ on $\mathbb{R}^N \backslash B_{R}(0)$.
{Hence,}
	$$
	\int_{\mathbb{R}^N \backslash B_{R}(0)}|\nabla u|^q d x<	
	\int_{\mathbb{R}^N \backslash B_{R}(0)}|\nabla u|^2 d x<\infty.
	$$
	Since ${u} \in C^{\infty}(\mathbb{R}^N)$, we deduce that
	$$
	\int_{B_{R}(0)}|\nabla u|^q d x<\infty,
	$$
	and so
	$$
	\int_{\mathbb{R}^N}|\nabla u|^q d x{=}	
	\int_{\mathbb{R}^N \backslash B_{R}(0)}|\nabla u|^q d x+\int_{B_{R}(0)}|\nabla u|^q d x<\infty.
	$$
	
	If $ \frac{2 N}{N+2}<q<2$, $N \geq 2 $, we compute the following integral directly
	$$
	\begin{aligned}
		\int_{B_{1}(0)}|\nabla u|^q d x & =\int_{B_{1}(0)}\left[\sum_{i=1}^N\left|\frac{-\frac{N}{2}|x|^{\frac{N}{2}-1} \log (|x|)+|x|^{\frac{N}{2}-1}}{|x|^N|\log (|x|)|^2} \cdot \frac{x_i}{|x|}\right|^2\right]^{\frac{q}{2}} d x \\
		& =\int_{B_{1}(0)}\left|\frac{|x|^{\frac{N}{2}-1}-\frac{N}{2}|x|^{\frac{N}{2}-1} \log (|x|)}{|x|^N|\log (|x|)|^2}\right|^q d x \\
		& \leq C_1 \left( \int_{B_{1}(0)} \frac{|x|^{-q-\frac{N}{2} q}}{|\log (|x|)|^{2 q}} d x+ \frac{|x|^{-q-\frac{N}{2} q}}{|\log (|x|)|^q} \right) d x \\
		& =C_2 \int_1^{\infty} \left(\frac{r^{-1+N-q-\frac{N}{2} q}}{|\log r|^{2 q}}  d r+ \frac{r^{-1+N-q-\frac{N}{2} q}}{|\log r|^q} \right)  d r<\infty,
	\end{aligned}
	$$
	
	Where $C_1$ and $C_2$ are positive constants.
	Since $\int_1^{\infty} \frac{r^\alpha }{|{\ln r}|^{\beta}} d r<\infty$ for all $\alpha<-1$, $\beta>0$ and since $-1+N-q-\frac{N}{2} q<-1$, we have
	
	$$
	\int_1^{\infty} \left(\frac{r^{-1+N-q-\frac{N}{2} q}}{|\log r|^{2 q}} \right) d r<\infty ,\qquad
 \int_1^{\infty} \left(\frac{r^{-1+N-q-\frac{N}{2} q}}{|\log r|^{ q}} \right) d r<\infty.
	$$
{Hence,}
	$\int_{\mathbb{R}^N}|\nabla u|^q d x<\infty$ holds.
\end{proof}

\subsection*{Conflict of interest}

The authors declare no conflict of interest.

\subsection*{Ethics approval}
 Not applicable.

\subsection*{Data Availability Statements}
Data sharing not applicable to this article as no datasets were generated or analysed during the current study.

\subsection*{Acknowledgements}
C. Ji was partially supported by National Natural Science Foundation of China (No. 12171152). P. Pucci is a member of the {\em Gruppo Nazionale per
l'Analisi Ma\-te\-ma\-ti\-ca, la Probabilit\`a e le loro Applicazioni} (GNAMPA) of the {\em Instituto Nazionale di Alta Matematica} (INdAM)
and this paper was written under the auspices of GNAMPA--INdAM.

\end{document}